\documentclass[11pt]{amsart} 
\usepackage{geometry}             

\usepackage{float}
\usepackage{fullpage}
\usepackage{mathabx}
\usepackage[mathscr]{euscript}
\usepackage[all]{xy}
\usepackage{epsfig}
\usepackage[T1]{fontenc}
\usepackage{amsfonts}
\usepackage{caption}

\usepackage{graphicx}
\usepackage{amssymb}
\usepackage{amsmath}
\usepackage{amsthm}
\usepackage{mathrsfs}
\usepackage{epstopdf}
\usepackage{url}

\usepackage{soul} 

\usepackage{upgreek}

\usepackage[dvipsnames,usenames,table]{xcolor}
\usepackage[colorlinks=true,urlcolor=blue,linkcolor=blue,citecolor=blue]{hyperref}
\hypersetup{
	linkbordercolor={1 0 0}, 
	citebordercolor={0 1 0} 
}
\usepackage{cite}
\usepackage[noabbrev]{cleveref}

\usepackage{pgfplots} 
\usepackage{tikz} 
\usetikzlibrary{knots}

\newcommand{\al}{\alpha}
\newcommand{\be}{\beta}
\newcommand{\ga}{\gamma}

\newcommand{\ep}{\varepsilon}

\renewcommand{\phi}{\varphi}
\newcommand{\si}{\sigma}

\newcommand{\ZZ}{{\mathbb Z}}
\newcommand{\CC}{{\mathbb C}}

\newcommand{\cE}{\mathcal E}

\newcommand{\cS}{\mathcal S}

\newcommand{\sS}{\mathfrak S}

\newcommand{\inv}{\operatorname{inv}}

\newcommand{\sm}{\smallsetminus}

\newcommand{\wh}[1]{\widehat{#1}}

\newcommand*\wbar[1]{
  \hbox{ \kern-0.2em%
    \vbox{%
      \hrule height 0.5pt  
      \kern0.25ex
      \hbox{%
        \kern-0.15em
        \ensuremath{#1}%
        \kern-0.05em
      }%
    }%
  \kern0.05em}%
} 

\newcommand{\KP}[1]{%
  \begin{tikzpicture}[baseline=-\dimexpr\fontdimen22\textfont2\relax]
  #1
  \end{tikzpicture}%
}
\newcommand{\KPY}{%
  \KP{
    \draw[->,color=black, thick] (-0.4,-0.4) -- (0.4,0.4);
    \draw[color=black, thick] (-0.4,0.4) -- (-0.05,0.05);
    \draw[->,color=black, thick] (0.05,-0.05) -- (0.4,-0.4);
    \draw[black, dashed] circle (0.58 cm);
  }%
}
\newcommand{\KPX}{%
  \KP{
    \draw[->,color=black, thick] (-0.4,0.4) -- (0.4,-0.4);
    \draw[color=black, thick] (-0.4,-0.4) -- (-0.05,-0.05);
    \draw[->,color=black, thick] (0.05,0.05) -- (0.4,0.4);
    \draw[black, dashed] circle (0.58 cm);
  }%
}
\newcommand{\KPA}{%
  \KP{%
    \draw[->,color=black, thick] (-0.4,-0.4) .. controls (0.02,0) .. (0.4,-0.4);
    \draw[->,color=black, thick] (-0.4,0.4) .. controls (-0.02,0) .. (0.4,0.4);
    \draw[black, dashed] circle (0.58 cm);
  }%
}
 
\makeatletter
\@namedef{subjclassname@2020}{\textup{2020} Mathematics Subject Classification}
\makeatother

\newtheorem{theorem}{Theorem}[section] 
\newtheorem{lemma}[theorem]{Lemma}
\newtheorem{proposition}[theorem]{Proposition}

\theoremstyle{definition}     
\newtheorem{definition}[theorem]{Definition}
\newtheorem{example}[theorem]{Example}

\theoremstyle{remark}

\title[Braid representatives minimizing the number of simple walks]{Braid representatives minimizing \\ the number of simple walks}
\author[H. U. Boden]{Hans U. Boden}
\address{Mathematics \& Statistics, McMaster University, Hamilton, Ontario}
\email{boden@mcmaster.ca}
\author[M. Shimoda]{Matthew Shimoda}
\address{Mathematics \& Statistics, McMaster University, Hamilton, Ontario}
\email{shimodm@mcmaster.ca}

\date{\today}                 
\begin{document}

\begin{abstract}
Given a knot, we develop methods for finding a braid representative that minimizes the number of simple walks. Such braids lead to an efficient method for computing the colored Jones polynomial of $K$, following an approach developed by Armond and implemented by Hajij and Levitt. We use this method to compute the colored Jones polynomial in closed form for the knots $5_2, 6_1,$ and $7_2$. The set of simple walks can change under reflection, rotation, and cyclic permutation of the braid, and we prove an invariance property which relates the simple walks of a braid to those of its reflection under cyclic permutation. We study the growth rate of the number of simple walks for families of torus knots. Finally, we present a table of braid words that minimize the number of simple walks for knots up to 13 crossings.
\end{abstract}

\maketitle


\subsection{Introduction} \label{S1}  \setcounter{section}{1} \setcounter{theorem}{0} 
The Jones polynomial $V_{L}(t)$ is an invariant of knots and links defined using quantum representations of braids.  It can be uniquely characterized as the polynomial-valued invariant of oriented links with $V_{\bigcirc}(t) = 1$ for $\bigcirc$ the unknot and satisfying the skein relation 
\[ t^{-1} V_{L_+}(t) - t V_{L_-}(t) = (t^{1/2} - t^{-1/2}) V_{L_0}(t),\]
where $L_+, L_-, L_0$ are identical outside a neighborhood, where they are as pictured 
\begin{equation*} 
\KPX \;\; L_+ \qquad \KPY \;\; L_- \qquad \KPA \;\; L_0.
\end{equation*}
The Jones polynomial admits a combinatorial state sum formula that can be used to compute it, but the complexity of the computation grows exponentially with the crossing number. So this method is impractical for computations that involve links with a large number of crossings. 

The \textbf{colored Jones polynomial} is a powerful knot invariant packaged as a sequence of Laurent polynomials $J_{N,K}(q)$ for $N\geq 2$, with $N=2$ giving the usual Jones polynomial. It encodes subtle geometric information about the knot complement and appears in several famous open problems in quantum topology, including (i) the Volume Conjecture \cite{Kashaev, Murakami-Murakami}; 
(ii) the Slope Conjecture \cite{Garoufalidis-Slope, Kalfagianni-Tran}; and (iii) the AJ Conjecture \cite{Garoufalidis-AJ}. The first relates the limit of $J_{N,K}(q)$ at $q=e^{2\pi i/N}$ as $N\to \infty$ to the hyperbolic volume of the knot complement;  the second posits that every Jones slope of a knot is the slope of an incompressible surface in the knot complement; and the third asserts that the recurrence relation for the $N$-th colored Jones polynomials is given by the $A$-polynomial of \cite{CCGLS}, a plane curve associated to the character variety of $SL(2,\CC)$ representations of the knot group. 

For further background information on the colored Jones polynomial and its relation to the geometry of 3-manifolds, we refer the reader to the books \cite{Murakami-Yokota} and \cite{FGP-2013} and their extensive bibliographies. 
The colored Jones polynomial also has intriguing number theoretic interpretations that will not be discussed in this paper;  for more details about these aspects, we refer the reader to the recent papers \cite{BGS-2021, Lovejoy-Osburn-2021, Lovejoy-Osburn-2019, BBMORTZ-2021} and their bibliographies.

Our goal in this paper is to study a probabilistic method for computing the colored Jones polynomial. This approach was first developed by Huynh and L\^{e} in \cite{Huynh-Le-2005}, and it was later described in terms of walks along braids by Armond \cite{Armond-2014}. Armond identified the special role played by \textit{simple walks}, resulting in an extremely efficient algorithm for computing $J_{N,K}(q)$, which has been implemented by Hajij and Levitt \cite{Hajij-Levitt-2018}. The algorithm is exponential in the number of simple walks on the braid, so it is natural to try minimize the number of simple walks before executing the program of \cite{Hajij-Levitt-2018}. However, as we shall see,  this number is highly dependent on the braid representative chosen. 

We study how the number of simple walks changes under taking reflection, rotation, and cyclic permutation of a given braid. We also examine the growth rate of the number of simple walks for two families of torus knots. For instance, for the family of $(2,n)$ torus knots, the simple walks satisfy a Fibonacci recurrence and grow exponentially in $n$. For the family of $(3,n)$ torus knots, the simple walks satisfy a tribonacci recurrence and also grow exponentially in $n$. We further prove that the total number of simple walks on a braid and its reflection is invariant under cyclic permutation. This fact is used to facilitate finding braid representatives with the least number of simple walks. For knots up to 13 crossings, we developed a program that finds minimal braid representatives. When these braids are used in conjunction with the program of Hajij and Levitt \cite{Hajij-Levitt-2018}, this provides an efficient method for computing the colored Jones polynomial for these knots.

We close this section with a brief synopsis of the rest of this paper. In \ref{S2}, we review the method from \cite{Huynh-Le-2005, Armond-2014} for  computing the colored Jones polynomial. In  \ref{S3}, we use it to compute $J_{N,K}(q)$ in closed form for the knots $5_2, 6_1,$ and $7_2$. These computations were originally performed by Masbaum using skein theory, see \cite{Masbaum-2003}. In \ref{S4}, we recall the basic results about braid representatives for knots and study the effect of the Markov moves on the set of simple walks.  In  \ref{S5}, we introduce the set of semi-simple walks, and we show that it is invariant under cyclic permutation of the braid word. In  \ref{S6} and \ref{S7}, we study the growth rate of number of simple walks for two families of torus knots. In  \ref{S8}, we present the output of a program for finding braid representatives that minimize the number of simple walks.

\subsection{The colored Jones polynomial and walks along braids} \label{S2}  \setcounter{section}{2} 
Given a knot $K$ and integer $N \geq 2$, the colored Jones polynomial $J_{N,K}(q)$ is a Laurent polynomial in the variable $q^{1/2}$. It is normalized so that $J_{N,\bigcirc}(q) = 1$, where $\bigcirc$ is the unknot. When $N=2$, the colored Jones polynomial agrees with the usual Jones polynomial. In general, the $N$-th colored Jones polynomial of a knot $K$ can be expressed in terms of the usual Jones polynomial of the $(N-1)$ strand cable of $K$. However, since the crossing number of the $(N-1)$ strand cable of a knot is $(N-1)^2$ times the crossing number of the knot, this does not lead to a practical method for computing the colored Jones polynomial.

One approach for computing the colored Jones polynomial is presented by Huynh and L\^{e} \cite{Huynh-Le-2005}. Starting with a braid $\be$ whose closure is the given knot, Huynh and L\^{e} use methods from quantum algebra to express the colored Jones polynomial as the inverse of the quantum determinant of an almost quantum matrix. The matrix is constructed through the product of Burau matrices, which we obtain from the crossing and orientation properties of $\be$.

A second approach is presented by Armond \cite{Armond-2014}. It is based on a probabilistic interpretation of the colored Jones polynomial and involves counting \textbf{walks along braids}. This method is closely related to the previous one, and in fact it  provides a visual representation of the quantum algebra approach. The idea is to view walks along the braid as traversing the strands of the braid from the bottom to the top and to record information about the crossings and their orientations as a product of operators. The end result is the same as that obtained by taking the quantum determinant of the deformation of Burau matrices,  but Armond's approach is more accessible and requires less background material on operator theory. One interesting aspect is that the complexity of the computation is sensitive to the choice of braid word, and this will be explored further in \ref{S4}.
For now, we focus on describing Armond's approach and the special role played by the \textit{simple walks}.

We begin by introducing a little terminology from braid theory. 

\begin{definition}
A \textbf{braid} is a set of $m$ strands running from top to bottom with no reversals in vertical direction. The strands may cross each other, but only two strands can participate at each crossing. 

Given a braid $\be$, a \textbf{braid word} is an expression of the form 
\[\be = \si_{i_1}^{\ep_1}\si_{i_2}^{\ep_2} \dots \si_{i_\ell}^{\ep_\ell},\]
where $\ep_i = \pm 1$ and $\si_i$ is a symbol.
Braid words are read from left to right, and braids are drawn from top to bottom. For $1 \leq i \leq m-1$, $\si_i$ represents the braid with one crossing where the $(i+1)$-st strand crosses over the $i$-th strand. The inverse $\si_i^{-1}$ represents the braid where the $i$-th strand crosses over the $(i+1)$-st strand.
\end{definition}

The braid word $\be = \si_{i_1}^{\ep_1}\si_{i_2}^{\ep_2} \dots \si_{i_\ell}^{\ep_\ell}$ has $\ell$ crossings, and we say it has \textbf{braid length} $\ell.$ If $\be$ is a braid on $m$ strands, we say it has \textbf{braid width} $m$. Note that braid words are not uniquely determined, applying a relation (see below) will alter the word without changing the braid. The writhe of a braid is defined to be the sum of the  signs on all its crossings. For example, the braid word above has writhe $w(\be) = \sum_{i=1}^\ell \ep_i.$

The braid group on $m$ strands is denoted $B_m$. Abstractly, it is the group with generators $\si_1,\ldots,\si_{m-1}$ and relations given by  (i) $\si_i \si_j = \si_j \si_i$ for $1\leq i,j \leq m-1$ with $|i-j|>1$ and  (ii) $\si_i \si_{i+1} \si_i = \si_{i+1} \si_i \si_{i+1}$ for $1\leq i \leq m-2.$ Relation (i) is called \textit{far commutativity} and (ii) is called the \textit{Yang-Baxter  relation} (or \textit{braid relation}). The group operation is given by concatenation of words or, equivalently, by stacking geometric braids, one on top of the other. 
 
Next, we introduce the notions of paths and walks along braids. A \textbf{path} starts at the bottom of the braid and traverses arcs of the braid, sometimes jumping down, until it reaches the top of the braid. If the path starts at strand $i$ on the bottom and ends at strand $j$ at the top, we say it is a path from $i$ to $j$. Whenever the path encounters a crossing, if it is on the overstrand, it is allowed to jump down to the undercrossing arc. If it is on the understrand, then it must stay on that strand. At each crossing the path encounters, a \textbf{weight} from the set  $\{a_{i,\ep_i}, b_{i,\ep_i}, c_{i,\ep_i} \mid i=1,\ldots, \ell \}$ is assigned. The weight will depend on the crossing, its sign, and the arcs traversed by the path at that crossing. For example, if the path jumps down at the $i$-th crossing, it is assigned the weight $a_{i,\ep_i}$. Otherwise, it is assigned the weight $b_{i,\ep_i}$ if the path traverses the understrand  and $c_{i,\ep_i}$ if it traverses the overstrand. Note that, at a given crossing, the path will follow the braid unless it jumps down there. The total weight of the path is the product of the weights of the crossings.

A walk $W$ along $\be$ consists of a set $J \subseteq \{1, \dots, m \}$, a permutation $\pi$ of $J$, and a collection of paths with exactly one path in the collection from $j$ to $\pi(j)$ for each $j \in J$. The weight of a walk is $(-1)(-q)^{|J|+\inv(\pi)}$ times the product of the weights of the paths, where $|J|$ is the cardinality of $J$ and $\inv(\pi)$ is the number of inversions in $\pi$, i.e., the number of pairs of elements in $J$ with $i < j$ and $\pi(i) > \pi(j)$.

The paths in a walk are ordered from left to right, using their starting strand at the bottom of the braid. This induces an ordering on the weights. The order of the weights is important, and that is because the operators associated to the weights are non-commuting. Operators based at different crossings do commute, but operators at the same crossing do not.  Thus, the effect of non-commutativity can be computed locally at each crossings of the braid.  Given a walk $W$ and a crossing $i$ of the braid, we use $W_{(i)}$ to denote the local weight of paths of $W$ through crossing $i$ in the given order. If no path of $W$ passes through crossing $i$, then we set the local weight $W_{(i)}=1$.
 
A stack of walks is any ordered collection $W_1 \cdots W_k$ of walks. Visually, this can be viewed as stacking the walks on top of one another with $W_1$ at the top and $W_k$ on the bottom.  The weight of the stack is the product of the weights of the paths in the appropriate order. If two paths belong to different walks, then the path in the higher walk is multiplied to the left of the path in the lower walk. If two paths belong to the same walk, then the path which begins to the left of the other path is said to be above and is multiplied to the left of the other path. Given a stack $\cS=W_1\cdots W_k$ and a crossing $i$ of the braid, we let $\cS_{(i)}$ be the local weight of the stack at $i$; it is equal to the product $(W_1)_{(i)}\cdots (W_k)_{(i)}$ of local weights of the walks of the stack at $i$ in the given order.

With walks along braids established, we now show how they can be used to compute the colored Jones polynomial.  Let $R =\ZZ[q,q^{-1}].$ Let $\hat{x}$ and $\tau_{x}$ be operators on the ring $R [x^{\pm 1},y^{\pm 1},u^{\pm 1}]$ given by $\hat{x}f(x,y,\dots) = xf(x,y,\dots)$ and $\tau_{x}f(x,y,\dots) = f(qx,y,\dots)$. The operators $\hat{y}$, $\hat{u}$, $\tau_{y}$, $\tau_{u}$ are defined similarly. We associate operators to each of the crossing weights using the formulas:
\[ a_{+} = (\hat{u} - \hat{y}{\tau{x}}^{-1}){\tau_{y}}^{-1} ,\hspace{1cm} b_{+} = {\hat{u}}^2, \hspace{1cm} c_{+} = \hat{x}{\tau_{y}}^{-2}{\tau_{u}}^{-1},  \]
\[ a_{-} = (\tau_{y} - {\hat{x}}^{-1}){\tau_{x}}^{-1}\tau_{u} ,\hspace{1cm} b_{-} = {\hat{u}}^2 ,  \hspace{1cm} c_{-} = {\hat{y}}^{-1}{\tau_{y}}^{-1}\tau_{u}. \]

By taking the summation of all walks on the braid and writing their weights in terms of the above operators, we obtain the operator $P$. Letting $P$ act on the constant 1 and making the substitutions $x = z$, $y = z$ and $u = 1$, we obtain a polynomial $\cE(P)$. Let $\cE_{N}(P)$ denote the polynomial obtained by making the substitution $z = q^{N-1}$ to $\cE(P)$.

The next result shows how to compute the colored Jones polynomial of a knot from its braid representative.  It was proved by Huynh and L\^{e} in \cite{Huynh-Le-2005} and appears as Theorem 2.3 in \cite{Armond-2014}. 

\begin{theorem}\label{Huynh-Le1}
Let $K$ be a knot obtained as the closure of a braid $\be \in B_m$. Its colored Jones polynomial is given by
\begin{equation} \label{eq-HL}
J_{N,K}(q) = q^{(N-1)(w(\be)-m+1)/2} \sum_{n=0}^{\infty} \cE_{N}(P^n),  
\end{equation} 
where the operator $P$ is the sum of the weights of the walks on $\be$ with $J \subseteq \{2, \dots, m \}$.
\end{theorem}

Stacks of walks are produced when we take the product of the weights of walks. This occurs in the step when we expand the operator $P$ to the power of $n$. 

Huynh and L\^{e} also gave a useful method for evaluating the terms $\cE_{N}(P^n)$ which avoids operator theory. The key result is the following lemma from \cite{Huynh-Le-2005} which computes $\cE_{N}(P^n)$ directly from the weights once they have been put in a preferred order. In the following, we suppress the dependence of the weights on the crossing and write $a_\pm, b_\pm, c_\pm$ instead of $a_{i,\pm}, b_{i,\pm}, c_{i,\pm}$.

\begin{lemma}\label{Huynh-Le2}
\begin{eqnarray*}
\cE_N(b_{+}^s c_{+}^r a_{+}^d) &=& q^{r(N-1-d)} \prod_{i=0}^{d-1} (1 - q^{N-1-r-i}) \\
\cE_N(b_{-}^s c_{-}^r a_{-}^d) &=& q^{-r(N-1)} \prod_{i=0}^{d-1} (1 - q^{r+i+1-N}) 
\end{eqnarray*}
\end{lemma}

We will apply this lemma to the local weights $\cS_{(i)}$ of each stack at each crossing. However, before we can apply \Cref{Huynh-Le2}, we must first put the local weights at a crossing into the preferred order. This can always be achieved using the following relations:
\[ a_{+}b_{+} = b_{+}a_{+},\hspace{0.90cm} a_{+}c_{+} = qc_{+}a_{+}, \hspace{0.90cm} b_{+}c_{+} = q^2c_{+}b_{+}  \] 
\[ \;\; a_{-}b_{-} = q^2b_{-}a_{-},\hspace{0.6cm} a_{-}c_{-} = q^{-1}c_{-}a_{-},\hspace{0.5cm} b_{-}c_{-} = q^{-2}c_{-}b_{-} \]

Once the local weights have been put into the preferred order at each crossing, \Cref{Huynh-Le1} and \Cref{Huynh-Le2}  can be applied locally at each crossing to compute the colored Jones polynomial. 

The computation is simplified by the observation that only \textbf{simple walks} contribute to the  colored Jones polynomial \cite[Lemma 2.5 (a)]{Armond-2014}. Here, a walk is said to be \textbf{simple} if no two paths intersect in the walk, otherwise it is \textbf{non-simple}. It turns out that non-simple walks occur in cancelling pairs, so for the purpose of computing $J_{N,K}(q)$, it is enough to consider only simple walks.

The computation is further simplified by the fact that, for any stack of walks, the evaluation of its weights will vanish if the walks in the stack traverse the same arc on $N$ or more different levels  \cite[Lemma 2.5 (b)]{Armond-2014}. This is extremely useful because it  reduces the complexity of the computation and guarantees that only finitely many terms of $\sum_{n=0}^\infty \cE_{N}(P^n)$ contribute to the $N$-th colored Jones polynomial. In particular, for a fixed $N$, there will always be an upper bound to the integers $n$ which need to be considered in the infinite sum \cref{eq-HL}. In practice, this bound can be determined by comparing the arcs traversed by the set of all simple walks for a given braid word. 


\subsection{The colored Jones polynomial in closed form} \label{S3}  
\setcounter{section}{3} \setcounter{theorem}{0} 
In this section, we apply \Cref{Huynh-Le1} and \Cref{Huynh-Le2} to compute $J_{N,K}(q)$, the colored Jones polynomial, for the knots $5_2$, $6_1$ and $7_2$. This is achieved by choosing favorable braid representatives, namely those with only a few simple walks.

In \cite{Masbaum-2003}, Masbaum uses skein theory to compute the colored Jones polynomial for all twist knots, a class which includes $5_2, 6_1,$ and $7_2$. More general calculations of the colored Jones polynomial for the double twist knots can be found in \cite{Lovejoy-Osburn-2021, Lovejoy-Osburn-2019}.
 
\begin{figure}[ht]
\centering
\includegraphics[scale=0.06]{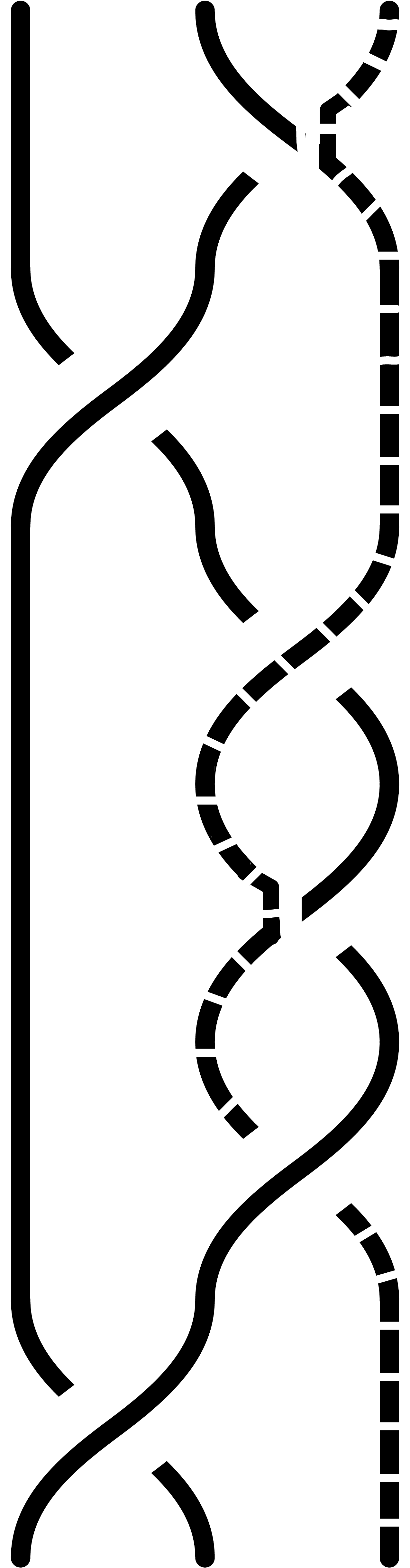}\hspace{1cm} \includegraphics[scale=0.06]{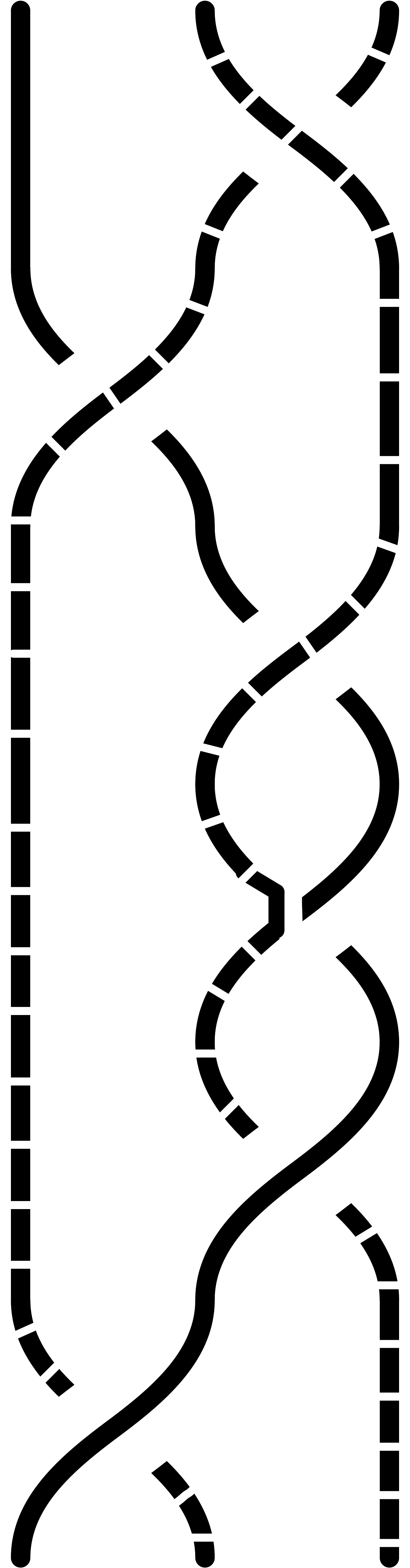} 
\vspace{-1mm}
\caption{\small The simple walks $A$ and $B$ shown as arcs with zebra stripes for the braid $\si_2^{-1} \si_1 \si_2^3 \si_1$ with closure the knot $5_2$.}\label{figure-2}
\end{figure}

\begin{example} \label{ex-5_2} 
The braid word $\si_2^{-1}\si_1 \si_2^3\si_1$ represents the knot $5_2$ and has two simple walks with $J \subseteq \{2,3\}$. They are $A = qa_{1,-}c_{3,+}a_{4,+}b_{5,+}$ and $B = q^3b_{1,-}c_{1,-}c_{2,+}c_{3,+}a_{4,+}b_{5,+}b_{6,+}$ (see \Cref{figure-2}). Notice that $A$ has $J=\{3\}$ and $B$ has $J=\{2,3\}$. Since the walks $A$ and $B$ both traverse the third strand at top and bottom,  stacks only need to be considered up to level $N-1$. 

Using \Cref{Huynh-Le1}, we can write the colored Jones polynomial as the following:
\begin{eqnarray*}
J_{N,K}(q) &=& q^{(1-N)} \sum_{n=0}^{N-1} \cE_{N}((A + B)^n), \\
&=& q^{(1-N)} \sum_{n=0}^{N-1} \cE_{N} ((qa_{1,-}c_{3,+}a_{4,+}b_{5,+} + q^3b_{1,-}c_{1,-}c_{2,+}c_{3,+}a_{4,+}b_{5,+}b_{6,+} )^n).  
\end{eqnarray*}

We will expand the above expression using the $q$-binomial theorem. For that purpose, we introduce the Gaussian binomial coefficients (or $q$-binomial coefficients), which are defined by  
\[ \binom{n}{k}_{\!\! q} = \prod_{i=0}^{k-1} \left( \frac{1-q^{n-i}}{1-q^{i+1}}\right). \]
The expansion of the above expression includes a sum of products of $A$'s and $B$'s, which can be interpreted as stacks. 

To apply \Cref{Huynh-Le2}, the weights at each crossing must be placed into the order $b^s c^r a^d$, and it is preferable to expand $(A+B)^n$ as a sum of terms of the form $B^k A^{n-k}$. Since the local weights at different crossings commute, the only potential issue with non-commutativity of $A$ and $B$ is at the first crossing. Since $a_{1,-}b_{1,-}c_{1,-} = qb_{1,-}c_{1,-}a_{1,-}$, we have $AB = qBA$, so we can adjust for inversions using the $q$-binomial coefficient:
\[(A + B)^n = \sum_{k=0}^n \binom{n}{k}_{\!\! q} B^kA^{n-k}.  \] 

Next we use \Cref{Huynh-Le2} to apply $\cE_{N}(\cdot)$ to evaluate each stack. First, these walks have weights $b_{i,+}$ indexed alone at the fifth and sixth crossings, which evaluates to $1$. Additionally, the $c_{i,\pm}$ weights at the first and second crossings in walk $B$ always cancel out since $\cE_{N}(c_{1,-}) = q^{-(N-1)}$ and $\cE_{N}(c_{2,+}) = q^{N-1}$. We still have $c_{3,+}$ in each walk $A$ and walk $B$. Therefore, for each walk in a stack, the term $A$ or $B$, $q^{N-1}$ is contributed. Similarly, the weight $a_{4,+}$ appears in both walks, so a stack consisting of $n$ walks will contribute $\prod_{i=1}^n (1-q^{N-i})$. Meanwhile, the weight $a_{1,-}$ is only in walk $A$, so the stack $B^k A^{n-k}$ will contribute $\prod_{i=1}^{n-k} (1-q^{n+i-N})$. Additionally, we need to adjust for the correct order of $b_{1,-}$ and $c_{1,-}$ in the weights in products containing $B^k$. The number of times the relation is applied increases quadratically with the exponent of $B$. That is, for a weight containing $B^k$, its contribution is $q^{k^2-k}$. Finally, the remaining powers of $q$ arise from the existing variables in the simple walks.

Applying \Cref{Huynh-Le1}, the colored Jones polynomial of $5_2$ can be written in closed form as
\[ J_{N,K}(q) = q^{N-1} \sum_{n=0}^{N-1}\sum_{k=0}^n \binom{n}{k}_{\!\! q} q^{nN+k(k+1)} \prod_{i=1}^n \left(1-q^{N-i}\right) \prod_{i=1}^{n-k} \left(1-q^{n+i-N}\right).  \]
\hfill $\Diamond$ \end{example}

\begin{figure}[ht]
\centering
\includegraphics[scale=0.08]{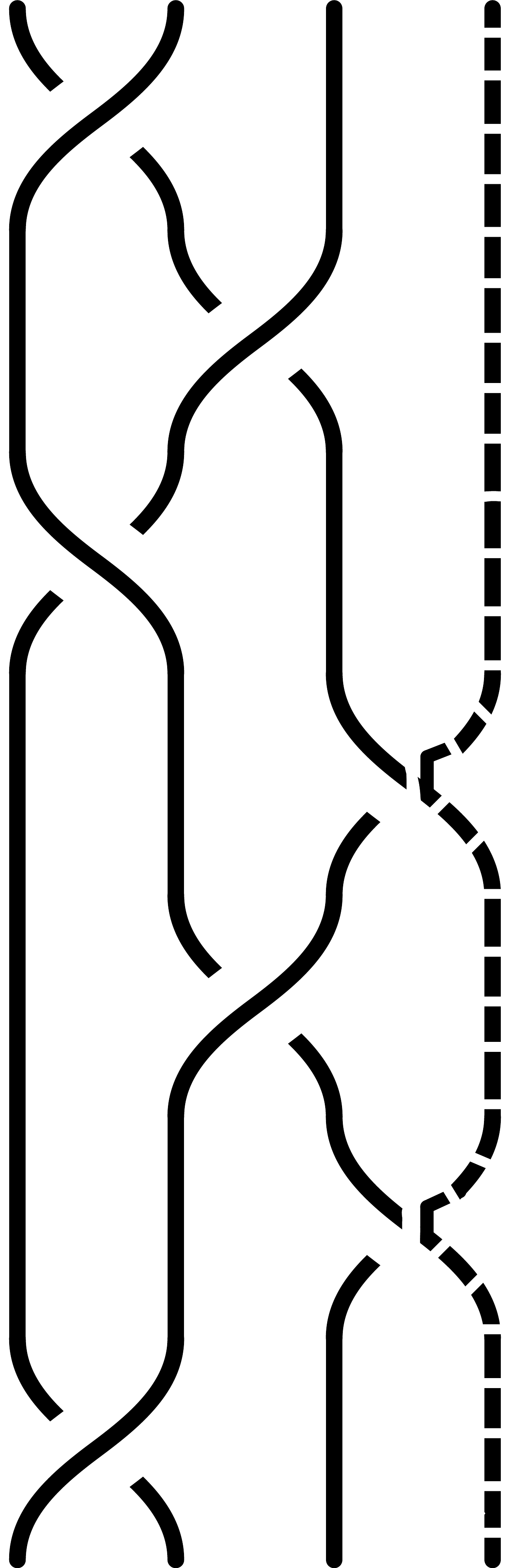}\hspace{1cm} \includegraphics[scale=0.08]{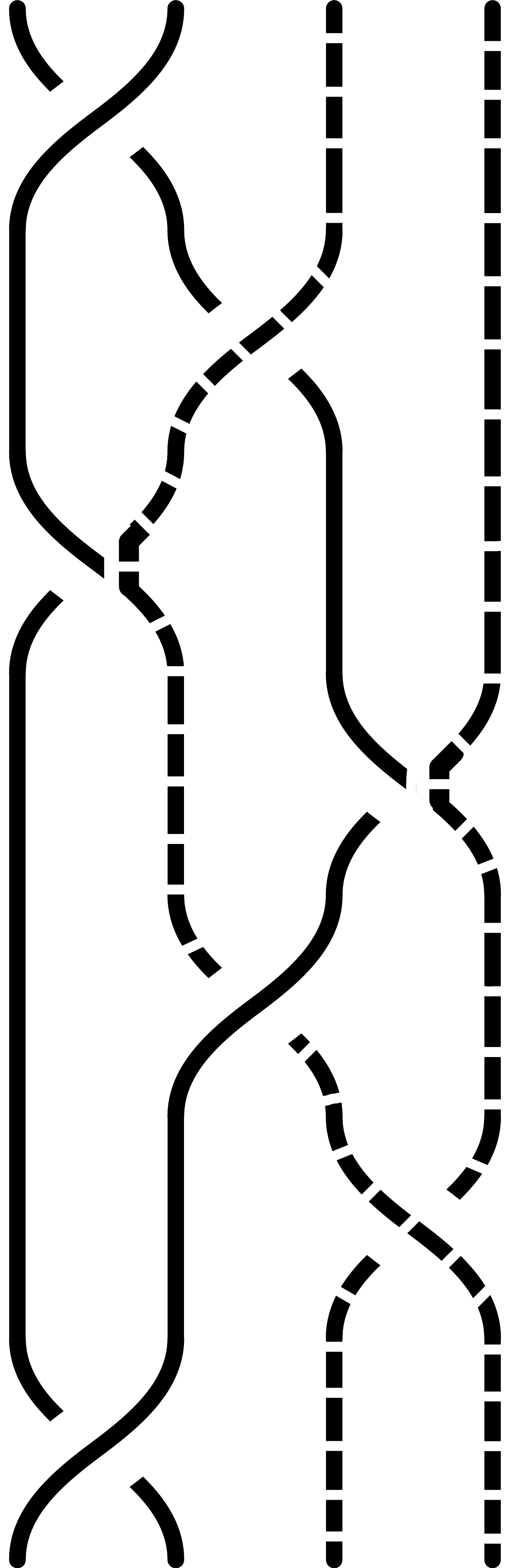}\hspace{1cm} \includegraphics[scale=0.08]{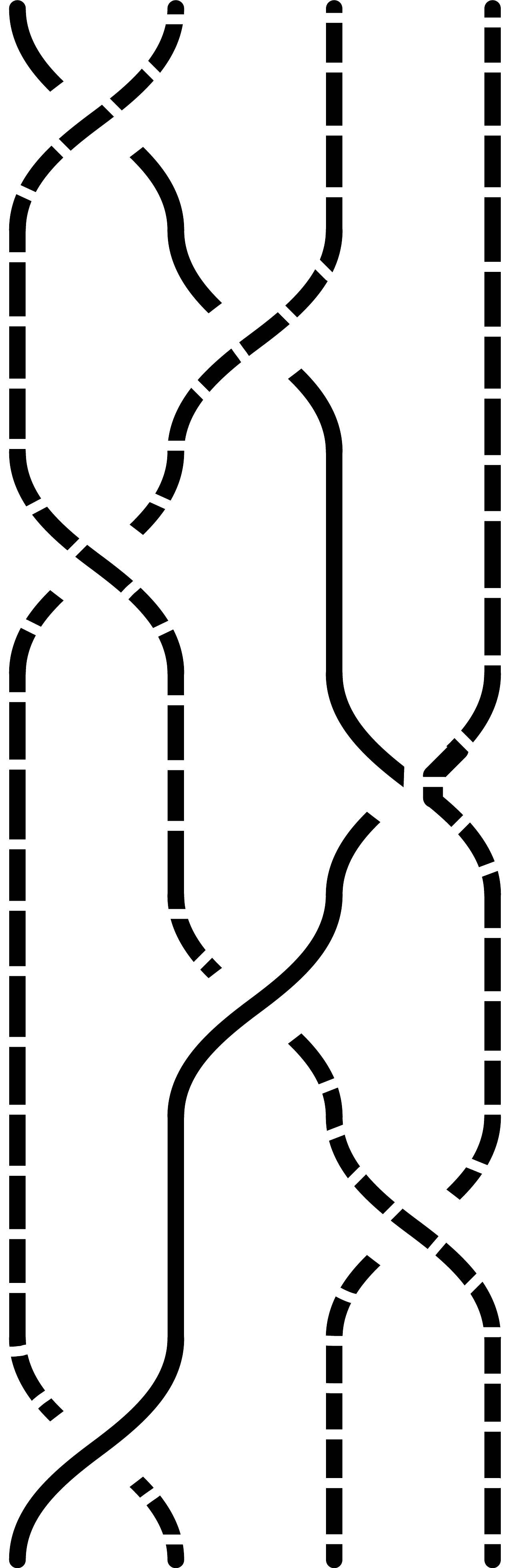} 
\vspace{-1mm}
\caption{\small The simple walks $A, B$ and $C$ shown as arcs with zebra stripes for the braid $\si_1\si_2{\si_1}^{-1}{\si_3}^{-1}\si_2{\si_3}^{-1}\si_1$ with closure the knot $6_1$.} \label{figure-3}
\end{figure}

\begin{example} \label{ex-6_1}
The braid word $\si_1\si_2{\si_1}^{-1}{\si_3}^{-1}\si_2{\si_3}^{-1}\si_1$ represents the knot $6_1$ and has three simple walks 
with $J \subseteq \{2,3,4\}$. They are  $A = qa_{4,-}a_{6,-}$, $B = q^3c_{2,+}a_{3,-}a_{4,-}b_{5,+}b_{6,-}c_{6,-}$ and $C = q^5c_{1,+}c_{2,+}b_{3,-}c_{3,-}a_{4,-}b_{5,+}b_{6,-}c_{6,-}b_{7,+}$ (see \Cref{figure-3}). Notice that $A$ has $J=\{4\},$ $B$ has $J=\{3,4\}$, and $C$ has $J=\{2,3,4\}$. Since the walks $A,B$ and $C$ all traverse the fourth strand at top and bottom,  stacks only need to be considered up to level $N-1$.

Using \Cref{Huynh-Le1}, the colored Jones polynomial for $6_1$ can be written as  
\begin{eqnarray*}
J_{N,K}(q) &=&  q^{(1-N)} \sum_{n=0}^{N-1} \cE_{N} \left((A + B + C)^n \right), \\
&=& q^{(1-N)} \sum_{n=0}^{N-1} \cE_{N} \left((qa_{4,-}a_{6,-} + q^3c_{2,+}a_{3,-}a_{4,-}b_{5,+}b_{6,-}c_{6,-}  \right.\\
&& \hspace{3cm} \left. + \; q^5c_{1,+}c_{2,+}b_{3,-}c_{3,-}a_{4,-}b_{5,+}b_{6,-}c_{6,-}b_{7,+})^n\right). 
\end{eqnarray*}

We use the $q$-multinomial theorem to expand the terms $(A+B+C)^n$, and to that end we recall the definition of the $q$-multinomial coefficients. 

\begin{figure}[ht]
\centering
\includegraphics[scale=0.08]{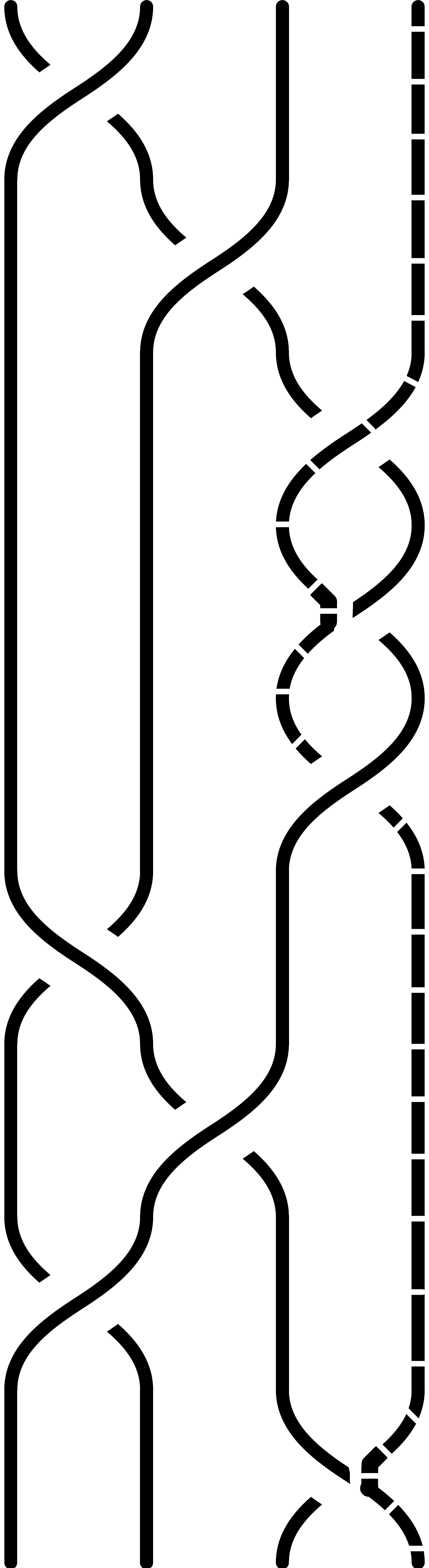}\hspace{1cm} \includegraphics[scale=0.08]{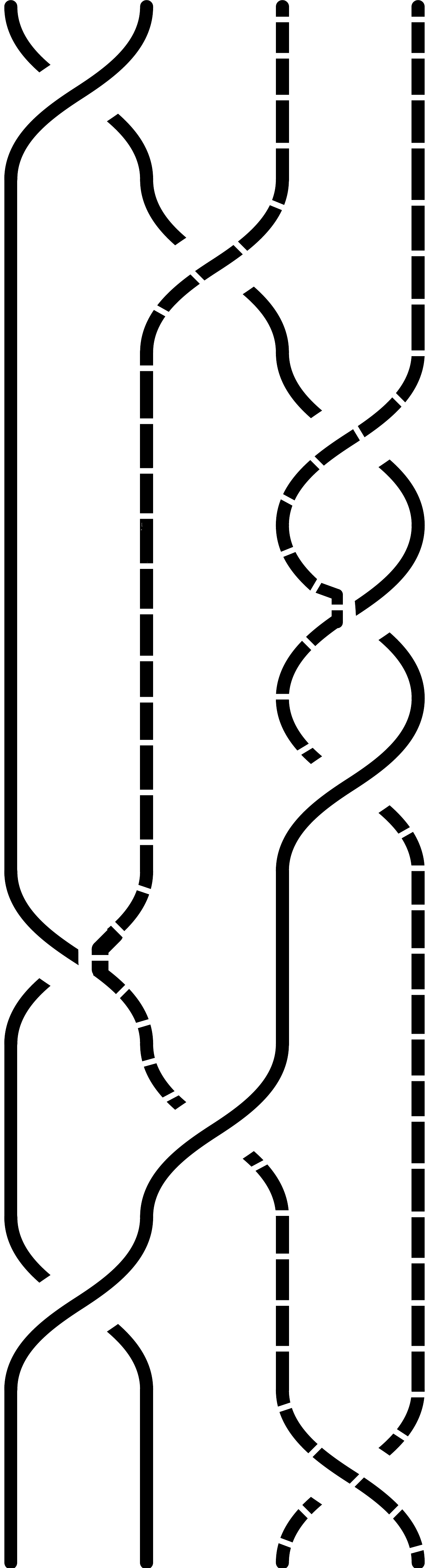}\hspace{1cm} \includegraphics[scale=0.08]{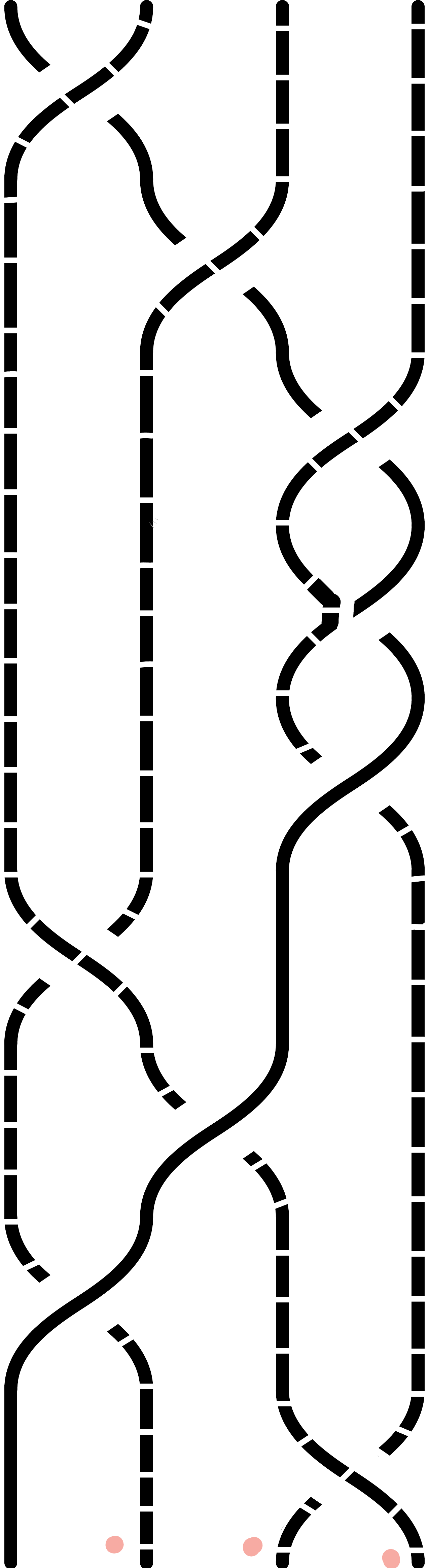} 
\vspace{-1mm}
\caption{\small The simple walks $A, B$ and $C$ as arcs with zebra stripes on the braid $\si_1\si_2 \si_3^3 \si_1^{-1}\si_2\si_1 \si_3^{-1}$ with closure the knot $7_2$.}
\label{figure-4}
 \end{figure}

Given an integer $r \geq 1$ and sequence of nonnegative integers $m_1,m_2,\dots,m_r$ such that $n = m_1+\cdots +m_r$, let
\begin{equation}\label{q-multinomial-coeff}
\binom{n}{m_1,m_2, \dots ,m_r}_{\!\! q} = \frac{[n]_q!}{[m_1]_q![m_2]_q!\dots[m_r]_q!}, 
\end{equation}
where $[n]_q = q^{n-1}+ \dots +q+1$ and $[n]_q! = [1]_q \cdots [n]_q$.  The term $\binom{n}{m_1,m_2, \dots ,m_r}_{\!\! q}$ is called the Gaussian multinomial coefficient or $q$-multinomial coefficient.

With the weights indexed at the third and sixth crossings, the natural order of the walks is $C,B,A$, as this will allow use of \Cref{Huynh-Le2} with only a few other adjustments. The use of the $q$-multinomial coefficient is suitable for this order because inversions with respect to the alphabet $C,B,A$ are reversed with multiplication by $q$. That is, $AB = qBA$, $AC = qCA$ and $BC = qCB$. The expansion of the trinomial takes the form 

\begin{equation}\label{eq-trinomial}
(A + B + C)^n = \sum_{m=0}^n \sum_{k=0}^{m} \binom{n}{n-m,m-k,k}_{\!\! q} C^{n-m}B^{m-k}A^k,  
\end{equation}
where $\binom{n}{n-m,m-k,k}_{\!\! q}$ denotes the $q$-trinomial coefficient of \cref{q-multinomial-coeff}.

We can now expand the trinomial for any power $n$ and use \Cref{Huynh-Le2} to apply $\cE_{N}(\cdot)$ and evaluate each stack. The $b_{\pm}$ and $c_{\pm}$ weights are evaluated the same as in \Cref{ex-5_2}. The weight $a_{4,-}$ appears in each of $A,B,C$, so for every term $C^{n-m}B^{m-k}A^k$, it will contribute $\prod_{i=1}^n (1-q^{i-N})$. Meanwhile, the weight $a_{3,-}$ only appears in the walk $B$, so for the term $C^{n-m}B^{m-k}A^k$, it contributes $\prod_{i=1}^{m-k} (1-q^{n-m+i-N})$. Similarly, the weight $a_{6,-}$ only appears in the walk $A$, so for the term $C^{n-m}B^{m-k}A^k$, it contributes $\prod_{i=1}^{m-k} (1-q^{n-k+i-N})$. 

Finally, we need to adjust for the correct order of the terms $b_{6,-}$ and $c_{6,-}$ in the products containing $C^{n-m}B^{m-k}$. The number of times the relation is applied increases quadratically with the sum of the exponents of $B$ and $C$, which is $(n-m)+(m-k)=n-k$. That is, for the term $C^{n-m}B^{m-k}A^k$, applying the relation introduces a factor of $q^{(n-k)^2-(n-k)}$. We follow the same logic for products $C^{n-m}$ to adjust for the order of the weights $b_{3,-}$ and $c_{3,-}$ at the third crossing.

Applying \Cref{Huynh-Le1}, the colored Jones polynomial of $6_1$ can be written in closed form as
\begin{multline*}
J_{K,N}(q) = q^{1-N} \sum_{n=0}^{N-1}\sum_{m=0}^n\sum_{k=0}^{m} \binom{n}{n-m,m-k,k}_{\!\! q} q^{3n-k-m+(n-k)^2+(n-m)^2} \\ 
\times \prod_{i=1}^n (1-q^{i-N}) \prod_{i=1}^{m-k} (1-q^{n-m+i-N}) \prod_{i=1}^k (1-q^{n-k+i-N}).
\end{multline*}
\hfill $\Diamond$
\end{example}

\begin{example} \label{ex-7_2} 
The braid word $\si_1\si_2 \si_3^3 \si_1^{-1}\si_2\si_1 \si_3^{-1}$ represents the knot $7_2$ and has three simple walks with $J \subseteq \{2,3,4\}$. They are $A = q c_{3,+} a_{4,+} b_{5,+} a_{9,-}$, $B = q^3 c_{2,+}c_{3,-}a_{4,+}b_{5,+}a_{6,-}b_{7,+} b_{9,-} c_{9,-} $ and $C = q^5c_{2,+}c_{3,+}a_{4,+}b_{5,+} b_{6,-}c_{6,-}b_{7,+} b_{8,-} b_{9,-} c_{9,-}$ (see \Cref{figure-4}). Notice that $A$ has $J=\{4\},$ $B$ has $J=\{3,4\}$, and $C$ has $J=\{2,3,4\}$. Since the walks $A,B$ and $C$ all traverse the fourth strand at top and bottom,  stacks only need to be considered up to $N-1$.

We use \Cref{Huynh-Le1} to write the colored Jones polynomial for $7_2$ as:
\begin{eqnarray*}
J_{N,K}(q) &=&  q^{(1-N)} \sum_{n=0}^{N-1} \cE_{N} \left((A + B + C)^n \right), \\
&=& q^{(1-N)} \sum_{n=0}^{N-1} \cE_{N} \left((qc_{3,+} a_{4,+} b_{5,+} a_{9,-} + q^3c_{2,+}c_{3,-}a_{4,+}b_{5,+}a_{6,-}b_{7,+} b_{9,-} c_{9,-} \right.\\
&& \hspace{3cm} \left. + \; q^5c_{2,+}c_{3,+}a_{4,+}b_{5,+} b_{6,-}c_{6,-}b_{7,+} b_{8,-} b_{9,-} c_{9,-})^n\right). 
\end{eqnarray*}

With the weights indexed at the sixth and ninth crossings, the most natural order of the walks is $C,B,A$. Note that $AB = qBA$, $AC = qCA$ and $BC = qCB$. Therefore, the expansion of the trinomial is as given in \cref{eq-trinomial} above.

We expand the trinomial for all powers of $n$ and use \Cref{Huynh-Le2} to apply $\cE_{N}(\cdot)$ and evaluate each stack. The $b_{\pm}$ and $c_{\pm}$ weights are evaluated the same as in Examples \ref{ex-5_2} and \ref{ex-6_1}. The weight $a_{4,+}$ appears in each of $A,B,C$, so for every term $C^{n-m}B^{m-k}A^k$, it will contribute $\prod_{i=1}^n (1-q^{N-i})$. Meanwhile, the weight $a_{6,-}$ only appears in the walk $B$, so for the term $C^{n-m}B^{m-k}A^k$, it contributes $\prod_{i=1}^{m-k} (1-q^{n-m+i-N})$. Similarly, the weight $a_{9,-}$ only appears in the walk $A$, so for the term $C^{n-m}B^{m-k}A^k$, it contributes $\prod_{i=1}^{k} (1-q^{n-k+i-N})$. 

Finally, we need to adjust for the correct order of the terms $b_{9,-}$ and $c_{9,-}$ in the products containing $C^{n-m}B^{m-k}$. The number of times the relation is applied increases quadratically with $(n-m)+(m-k) = n-k.$ That is, for the term $C^{n-m}B^{m-k}A^k$, applying the relation introduces a factor of $q^{(n-k)^2-(n-k)}$. We follow the same logic for products $C^{n-m}$ to adjust for the weights at the sixth crossing.

Applying \Cref{Huynh-Le1}, the colored Jones polynomial of $7_2$ can be written in closed form as
\begin{multline*}
J_{K,N}(q) = q^{N-1} \sum_{n=0}^{N-1}\sum_{m=0}^n\sum_{k=0}^{m} \binom{n}{n-m,m-k,k}_{\!\! q} q^{nN+2n-m-k+(n-k)^2 +(n-m)^2} \\ 
\times \prod_{i=1}^n (1-q^{N-i}) \prod_{i=1}^{m-k} (1-q^{n-m+i-N}) \prod_{i=1}^k (1-q^{n-k+i-N}).
\end{multline*}\hfill $\Diamond$
\end{example}

For hyperbolic knots, the Volume Conjecture asserts that   
\[ \lim_{N \to \infty} \frac{\log |J_{N,K}(e^{2 \pi i/N})|}{N} = \frac{\text{Vol}(S^3 \sm K)}{2 \pi}.\]
For more information about this important open problem, see the book \cite{Murakami-Yokota}. The knots in Examples \ref{ex-5_2}, \ref{ex-6_1}, and \ref{ex-7_2} are all hyperbolic, and the Volume Conjecture has been verified for each of them.   For $5_2$, this was proved by Ohtsuki \cite{Ohtsuki-2016};  for $6_1,$ it was proved by Ohtsuki and Yokota \cite{Ohtsuki-Yokota}; and for $7_2$, it follows from Ohtsuki's work on 7-crossing hyperbolic knots \cite{Ohtsuki-2017}.
It would be interesting to use the formulas given here to independently verify the Volume Conjecture for these knots.


\subsection{Representing knots as braids} \label{S4} 
\setcounter{section}{4} \setcounter{theorem}{0} 
One of the objectives in this paper is to find braid representations of knots that minimize the number of simple walks. To do that, we will apply several different operations that alter the braid word without changing its representative knot. These operations include reflection, rotation, and cyclic permutation of the braid word, and each of them can be used to reduce the number of simple walks. We begin with a review of some standard material on representing knots as braids (see also \cite{Birman}).

\vspace{-1.5cm}
 \begin{figure}[ht]
  {
	\begin{minipage}[c][1\width]{
	   0.35\textwidth}
	   \centering
	  \quad  \includegraphics[scale=0.95]{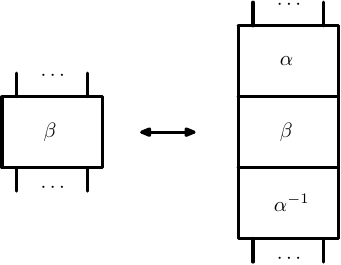}
	   \caption*{\small Conjugation}
	\end{minipage}}
 \hfill 	
  {
	\begin{minipage}[c][1\width]{
	   0.55\textwidth}
	   \centering
	   \includegraphics[scale=0.95]{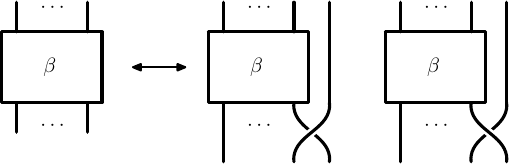} \quad
	   \caption*{\small Stabilization}
	\end{minipage}}
 \vspace{-2.00cm}
 \caption{\small The Markov moves.} \label{figure-1}
\end{figure}

Given a braid $\be$, its closure is denoted $\wh{\be}$ and is the knot or link obtained by connecting the strands on top with the corresponding strands on bottom without introducing any additional crossings. 

The next result is called Alexander's theorem and was first proved in 1923, see \cite{Alexander-1923}.
\begin{theorem} 
Every oriented  knot or link is equivalent to the closure $\wh{\be}$ for some braid $\be \in B_m$.
\end{theorem}

\begin{definition} 
The \textit{Markov moves} include \textit{conjugation} and \textit{stabilization} (see \Cref{figure-1}). Given a braid $\be \in B_m$, \textbf{conjugation} involves replacing it with $\al \be \al^{-1}$ for some $\al \in B_m$. \textbf{Stabilization} involves replacing $\be$  with either  $\be \si_m$ or $\be \si_m^{-1}$. Note that conjugation preserves the braid width and stabilization increases it by one. 
\end{definition}

The next result is attributed to Markov. For a proof, see  \cite[Theorem 2.3]{Birman}.
\begin{theorem} Two braids have equivalent link closures if and only if they are related by a sequence of Markov moves. 
\end{theorem}

Let $\be \in B_m$, and let 
\begin{equation} \label{eq:SW}
SW_{\be} =\{W \mid W \text{ is a simple walk on $\be$ with } J \subseteq \{2,\ldots, m\} \}
\end{equation}
be the set of simple walks on the braid $\be$.

Given a knot, we are interested in finding the braid representative that minimizes the number of simple walks. Of course, the set of simple walks depends on the braid representative chosen. In fact, it depends on the braid word since it is not preserved under insertion (or deletion) of $\si_i \si_i^{-1}$ or $\si_i^{-1} \si_i$ into the braid word. This is the analogue, for braids, of the Reidemeister II move. We explain this important point next.

Given a walk on a braid, we say that an arc of the braid is \textbf{active} if it is traversed by a path of the walk.  Similarly, we say that a crossing is \textbf{active} if the walk jumps down from overcrossing arc to undercrossing arc at that crossing. Thus, the active crossings are the ones with the local weight $a_{i,\pm}.$
In Figures \ref{figure-2}, \ref{figure-3}, \ref{figure-4}, the active arcs are depicted with zebra stripes.

Now consider two braid words: $\ga = \al \be$ and $\ga' = \al \si_i \si_i^{-1} \be$. (A similar argument applies to $\ga' = \al \si_i^{-1}\si_i  \be$.) We will show that  $SW_\ga \subseteq SW_{\ga'}$. 

Suppose $W$ is a simple walk on $\ga$. If both strands $i,i+1$ are active or if they are both not active, then $W$ extends in a unique way to a simple walk on $\ga'$. If one of the strands $i,i+1$ is active and the other is not, then $W$ extends to a simple walk on $\ga'$, but possibly in more than one way. This proves the claim, and in particular, we see that the number of simple walks is non-decreasing under an elementary insertion. 

Recall that a braid word is said to be \textbf{reduced} if it does not contain an occurrence of $\si_i \si_i^{-1}$ or $\si_i^{-1} \si_i$. By the above considerations, for any given knot, we can always assume that its braid representative is given by a reduced word.

In a similar way, one can show that the set of simple walks is invariant under far commutativity and the Yang-Baxter relation. For far commutativity, this is straightforward, and we leave the details to the reader. For the Yang-Baxter relation, consider the braid words $\ga = \al \si_i \si_{i+1} \si_i \be$ and $\ga' = \al \si_{i+1} \si_i \si_{i+1} \be$ and  assume the relevant crossings are $j, j+1,$ and $j+2.$

We claim that any simple walk on $\ga$ extends in a unique way to a simple walk on $\ga'$. There are several cases, depending on which of the three crossings $j,j+1,j+2$ are active. If none of the crossings are active, then it extends to a simple walk on $\ga'$. If one of the crossings is active, and if we make the corresponding crossing on $\ga'$ active, and then it extends to a simple walk on $\ga'$. If two of the crossings of $\ga$ are active, then they must be $j$ and $j+2,$ and it extends to a simple walk on $\ga'$ again with $j$ and $j+2$ active crossings. Note that it is not possible for all three crossings to be active. This shows that $|SW_\ga| = |SW_{\ga'}|$ under the Yang-Baxter relation.

One can also apply the Markov moves to a braid and consider their effect on the set of simple walks. For instance, under conjugation, one would expect that the resulting braid word will have a larger set of simple walks. A special case is \textbf{cyclic permutation}, which involves replacing $\be=\si_{i_1}^{\ep_1}\si_{i_2}^{\ep_2} \cdots \si_{i_\ell}^{\ep_\ell}$ with $\be' = \si_{i_2}^{\ep_2} \cdots \si_{i_\ell}^{\ep_\ell} \si_{i_1}^{\ep_1}.$ We will study the effect of  cyclic permutation on the set of simple walks in the next section.

Under stabilization, we will show that the set of simple walks is non-decreasing. Let $\be \in B_m$, $\be' = \be \si_m^{\pm 1} \in B_{m+1}$, and suppose $W$ is a simple walk on $\be$. If $m \not\in J,$ then $W$ extends uniquely to a simple walk on $\be'$. If $m \in J$ and $\be' = \be \si_m,$ then we can extend $W$ to a simple walk on $\be'$ by either making the extra crossing active or by setting $J' = J \cup \{m+1\}$. If $m \in J$ and $\be' = \be \si_m^{-1},$ then we can extend $W$ to a simple walk on $\be'$ with $J' = J \cup \{m+1\}.$ Notice that for $\be' = \be \si_m^{-1}$, we have one additional simple walk that does not come from $\be$, namely the one with $J = \{m+1\}$ and the extra crossing made active.  In particular, it follows that $|SW_\be| \leq |SW_{\be'}|.$

Let $K$ be a knot and suppose $\be \in B_{m+1}$ is a braid representative for $K$ with $m\geq 1.$ If $\be$ is conjugate to a braid of the form $\ga \si_{m}^\pm$ for some braid $\ga \in B_{m},$ then $\be$ is said to be  \textbf{reducible}. The braid $\be$ is said to be \textbf{irreducible} if it is not reducible. 

The next result summarizes our discussion.
\begin{proposition} \label{prop:reduced-irreducible}
If $K$ is a knot, then any braid representative for $K$ that minimizes the number of simple walks can be assumed to be given by a reduced and irreducible braid word.
\end{proposition}

In addition, one can apply symmetry operations to alter the braid word without changing its knot or link closure. We will use these operations to find braid representatives that minimize the number of simple walks. The three operations we will consider are called reflection, rotation, and reversal, and we introduce them next.

For a given braid $\be$, its reflection is denoted $\be^*$ and is obtained by switching all the crossings of $\be$. If $\be$ represents the knot $K$, then $\be^*$ represents its mirror image $K^*$. If $\be = \si_{i_1}^{\ep_1} \cdots \si_{i_\ell}^{\ep_\ell}$, then its reflection is the braid word given by $ \be^* = \si_{i_1}^{-\ep_1} \cdots \si_{i_\ell}^{-\ep_\ell}$ (see \Cref{Fig-braid-symmetry}). 

For the purposes of computing the colored Jones polynomial, one can use either $\be$ or $\be^*$, since the invariants are related by the simple formula $J_{N,K}(q) = J_{N,K^*}(q^{-1}).$ The two braids will have completely different sets of simple walks. In fact, as we shall see in the next section, the simple walks on $\be$ and $\be^*$ are disjoint and complementary to one another. There is an obvious computational advantage to working with the braid having fewer simple walks.


In fact, there are  other symmetries that can be applied to get a new braid representative for a knot (or its mirror image). For example, given a braid word $\be$ representing a knot, if one rotates it  180$^\degree$  in the plane, one obtains a new braid word representing the same knot.  Specifically, if $\be = \si_{i_1}^{\ep_1} \cdots \si_{i_\ell}^{\ep_\ell}$, then the rotated braid word is denoted $ \be^\dag$ and is given by $\be^\dag = \si_{m-i_\ell}^{\ep_\ell} \cdots \si_{m-i_1}^{\ep_1}$  (see \Cref{Fig-braid-symmetry}).

Another example is braid reversal, which is given by reversing the order of the braid word. Again, the new braid represents the same knot. If $\be = \si_{i_1}^{\ep_1} \cdots \si_{i_\ell}^{\ep_\ell}$, then its reversal is denoted $\be^r$ and is given by $\be^r =  \si_{i_\ell}^{\ep_\ell} \cdots \si_{i_1}^{\ep_1}$ (see \Cref{Fig-braid-symmetry}). Notice that $\be^r$ is the braid obtained from $\be$ by rotating it 180$^\degree$ around a horizontal line in the plane.

There is a one-to-one correspondence between the sets of simple walks on $\be$ and $\be^r$.  Under the correspondence, the walks   have the same set of active crossings, and the weights for the over- and undercrossings $(b_{i,\pm}, c_{i,\pm})$ are switched. In fact, as we shall see, a simple walk is completely determined by its set of active crossings, and it follows that $|SW_\be| = |SW_{\be^r}|.$

\begin{figure}[ht]
\centering
\includegraphics[scale=0.6]{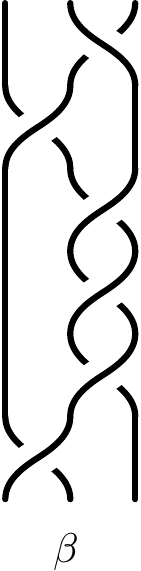} \hspace{1cm}
\includegraphics[scale=0.6]{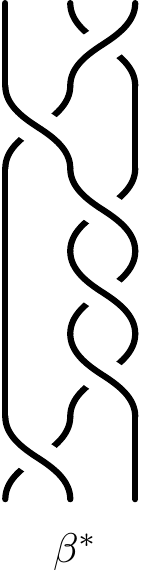} \hspace{1cm}
\includegraphics[scale=0.6]{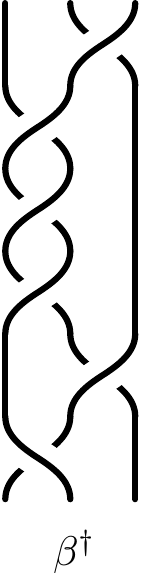} \hspace{1cm}
\includegraphics[scale=0.6]{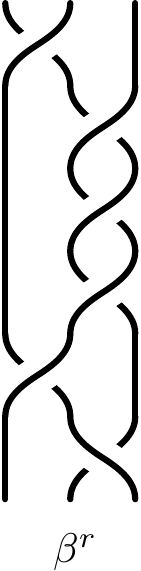}
\caption{\small A braid $\be$ representing the knot $5_2$ and its reflection $\be^*$, rotation $\be^\dag$, and reversal $\be^r$.}
\label{Fig-braid-symmetry}
\end{figure}

Given a braid word for a knot, applying reflection, rotation, or cyclic permutation will alter its set of simple walks. Since the computation of the colored Jones polynomial is exponential in the number of simple walks, it is advantageous to choose the braid representative that minimizes the set of simple walks.


\subsection{Semi-simple walks and cyclic permutation} \label{S5} 
\setcounter{section}{5} \setcounter{theorem}{0} 
The main result in this section is an invariance property which asserts that under the cyclic permutation, the total number of simple walks on a braid $\be$ and its reflection $\be^*$ does not change. To see this, we introduce the notion of semi-simple walks and study their behavior under reflection.   

Recall the definition of $SW_{\be}$ in \cref{eq:SW}. Previously,  we identified walks $W$ with their weights, given by the ordered product of operators $\{a_{i,\pm}, b_{i,\pm}, c_{i,\pm}\}$ for each crossing traversed. However, it will be more convenient to record $W$ using only the operators $a_{i,\pm}$, and we can do so with no loss of information. In the following, we write $a_i$ instead of $a_{i,\pm}$; it is notationally more compact and the sign $\pm$ can be recovered from the braid word. Thus, there is a one-to-one correspondence between simple walks on $\be$ and (certain) monomials in $\{a_1,\ldots, a_\ell\}$, as we shall now explain. 

Given a simple walk $W$, recall that the active crossings are where the walk jumps down from the overcrossing arc to the undercrossing arc. If the $i$-th crossing is active, we record this with $a_i$. As usual, the crossings are labeled $1,2,\ldots, \ell$ from top to bottom of the braid. The collection of active crossings of $W$ determines a monomial in $\{a_1,\ldots, a_\ell\}$, and thus we see that a simple walk determines a monomial. Conversely, the monomial in $\{a_1,\ldots, a_\ell\}$ uniquely determines the simple walk $W$. We will explain this below, but before we do, notice that not every monomial corresponds to a simple walk. For example, the trefoil braid $\si_1^3$ has three crossings and so there are $2^3=8$ possible monomials. However, it has only one simple walk corresponding to the monomial $a_2$.

Suppose then that $a_{i_1}\cdots a_{i_k}$ is a monomial, indicating that the crossings $i_1,\ldots, i_k$ are active. We perform an oriented smoothing at each active crossing. Since the walk jumps down there, the crossing type determines which of the arcs are active and which are not. Specifically, if the crossing is positive, the active arc is the one on the left, and if the crossing is negative, the active arc is the one on the right (see \Cref{figure-5}). At each active crossing, we mark the active arc using some marking scheme. (In all figures, the active arcs have zebra stripes.) We then extend the marking along the arc through any inactive crossings and around the back of the braid closure, continuing again through the braid and around the back as many times as necessary, until reaching another active crossing.

\begin{figure}[ht]
\centering
\includegraphics[scale=0.90]{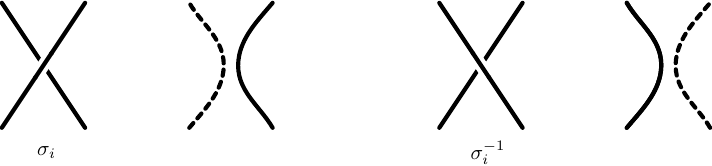}
\vspace{-1mm}
\caption{\small An active crossing and its oriented smoothing for $\si_i$ and $\si_i^{-1}$. The active arc has zebra stripes and the inactive arc is solid}.
\label{figure-5}
 \end{figure}

One of two things will happen at this active crossing. Either the extended marking is on an inactive arc, in which case the monomial does not correspond to a valid simple walk, or it is on the active arc. If the second case holds for all extended arcs, then this is a valid simple walk. Notice that, in that case, every arc of the partial smoothing of the braid closure is connected to either an active or an inactive arc. To see that we argue by contradiction. If the partial smoothing of $\wh{\be}$ contains an arc that is not connected to an active or an inactive arc, then we can follow it along $\wh{\be}$ and it will only pass through inactive crossings before returning to itself. Therefore, it determines a sublink of the closure of $\be$, which contradicts the assumption that the closure of $\be$ is a knot.

Note that, in the second case, the simple walk is determined by the active or marked arcs, and once those are specified we can read the full operator by recording all the crossings it passes through. This explains why the simple walk is determined by its active crossings. Note that, in order for the marked arcs to be a simple walk, the first strand of the braid at the top and bottom must be inactive.

This process can alternatively be understood in terms of taking the partial smoothing of a knot $K$ at a subset of crossings. Given a subset of crossings, we can perform the partial oriented smoothing of $K$ at the selected crossings. In general, this will produce a link. Our assumption is that the resulting link can be partitioned into two sublinks, one containing the active (or marked) arcs the other containing the inactive (or unmarked) arcs.

So, for a given monomial to correspond to a simple walk, it must be the case that the active and inactive arcs of the braid are contained in different components of the partial smoothing. Further, it must be the case that the first strand of the braid (at bottom) is inactive.

We can apply this to understand the behavior of simple walks under taking reflection.  In the mirror image, all the crossings are switched. So, using same monomials to record simple walk on $\be$ and $\be^*$, it follows that taking reflection is equivalent to switching the active and inactive arcs. That is because each positive crossing of $\be$ is negative in $\be^*$ and vice versa. This is illustrated by switching from left to right or vice versa in \Cref{figure-5}. 

The walks on $\be$ and $\be^*$ with the same monomial are \textit{dual}.  We explain this from the point of view of the full set of weights. Let $W$ be a walk on $\be$ with monomial $a_{i_1}\cdots a_{i_k}$, and let $W^*$ be the corresponding walk on $\be^*$ with monomial $a_{i_1}\cdots a_{i_\ell}$. Then $W$ and $W^*$ trace out disjoint arcs of the braid obtained from $\be$ by smoothing the crossings $i_1,\ldots,i_k$. In terms of the weights, the walks $W$ and $W^*$ have the same set of active crossings, but at the inactive crossings, their local weights are opposite. Specifically, if $i$ is an inactive crossing and the local weight $W_{(i)} =b_{i}$, then $W^*_{\; (i)} = c_{i}$. If instead $W_{(i)} =c_{i}$, then $W^*_{\; (i)} = b_{i}$. Likewise, if $W_{(i)} =1$, then  $W^*_{\; (i)} =  b_{i}c_{i}$, and if $W_{(i)} =b_{i}c_{i}$ then $W^*_{\;(i)} =1$. 

\begin{lemma}
The simple walks with $J \subseteq \{2, \ldots, m\}$ on a braid $\be$ and its reflection $\be^*$ are disjoint. In other words, $SW_\be \cap SW_{\be^*} = \varnothing$.
\end{lemma}

\begin{proof} 
Suppose $W$ is a simple walk on $\be$ with monomial $a_{i_1} \cdots a_{i_k}$. Then the partial smoothing of $\be$ at the crossings $i_1,\ldots, i_k$ can be partitioned into active and inactive arcs. Here, we mark the active arcs, with the first strand of the braid at the top inactive. For the same monomial on the mirror image  $\be^*$, the active and inactive arcs will be switched. In particular, the first strand on $\be^*$ will be active at the top and marked as such. Therefore, the monomial $a_{i_1} \cdots a_{i_k}$ will not correspond to a valid simple walk on $\be^*$. It follows that $SW_\be \cap SW_{\be^*} = \varnothing$, and this completes the proof.
\end{proof}

\begin{definition} 
Given a braid word $\be$, we say that a walk $W$ on $\be$ is \textbf{semi-simple} if it is a simple walk on $\be$ or on $\be^*$ with $J \subseteq \{2, \ldots, m\}$. We use $\sS_{\be}$ to denote the set of semi-simple walks on $\be$. Therefore, $\sS_{\be} = SW_{\be} \cup SW_{\be^*}$. 
\end{definition} 
Since $SW_{\be}$ and $SW_{\be^*}$ are disjoint, it follows that $|\sS_{\be}| = |SW_{\be}| + |SW_{\be^*}|,$ where $|S|$ denotes the cardinality of the finite set $S$. 

We leave it as an exercise to show that every monomial $a_i$ for $1\leq i \leq \ell$ corresponds to a simple walk on either $\be$ or $\be^*$. Thus $|\sS_{\be}| \geq n$.   

\begin{theorem} \label{thm-conservation}
The set of semi-simple walks $\sS_{\be}$ is invariant under cyclic permutation of the braid word.
\end{theorem}

The theorem is a direct consequence of the next two lemmas. The first lemma implies that cyclic permutation of $\be$ does not alter the set of simple walks unless $\be$ starts with $\si_1$ or $\si_1^{-1}.$

\begin{lemma} \label{lemma-1}
Suppose $\be = \si_{i_1}^{\ep_1}\si_{i_2}^{\ep_2} \cdots \si_{i_\ell}^{\ep_\ell}$ is a braid word with $i_1 \neq 1.$ Let $\be' = \si_{i_2}^{\ep_2} \cdots \si_{i_\ell}^{\ep_\ell}\si_{i_1}^{\ep_1}$ be the braid obtained by cyclic permutation. Then $SW_\be =SW_{\be'}$.
\end{lemma}

\begin{proof}
Suppose $W\in SW_\be$ is a simple walk on $\be$ with $J \subseteq \{2,\dots, m\}$.  Let $W'$ be the corresponding simple walk on $\be'$, with underlying set $J'$. There are three possible cases, depending on whether $i_1$ and $i_1+1$ lie in $J$. First, if neither $i_1$ nor $i_1+1$ lie in $J$, then cyclic permutation has no effect and $W'$ is a simple walk on $\be'$ with $J' =J$. Second, if exactly one of $i_1, i_1+1$ lies in $J$, then $J'\neq J$, but $W'$ is nevertheless a simple walk on $\be'$ with $J' \subseteq \{2,\dots, m\}$. Third, if both $i_1$ and $i_1+1$ are in $J$, then $J'=J$ and $W$ is a valid simple walk on $\be'$. This completes the proof of the lemma.
\end{proof}

The second lemma studies the effect of cyclic permutation for braids that start with $\si_1$ or $\si_1^{-1}.$ We will show that cyclic permutation of a braid $\be$ has the potential to exchange simple walks between $SW_{\be}$ and $SW_{\be^*}$, but it does not alter the set of semi-simple walks.

\begin{lemma} \label{lemma-2}
Suppose $\be = \si_{i_1}^{\ep_1}\si_{i_2}^{\ep_2} \cdots \si_{i_\ell}^{\ep_\ell}$ is a braid word with $i_1=1.$ Let $\be' = \si_{i_2}^{\ep_2} \cdots \si_{i_\ell}^{\ep_\ell}\si_{i_1}^{\ep_1}$ be the braid obtained by cyclic permutation. Then $\sS_\be =\sS_{\be'}$.
\end{lemma}

\begin{proof}
Suppose $W\in SW_\be$ is a simple walk on $\be$ with $J \subseteq \{2,\ldots, m\}$. Let $W'$ be the corresponding simple walk on $\be'$, with underlying set $J'$. There are two possible cases, depending on whether or not $J$ contains $2.$ If $2 \not\in J,$ then  $J'=J$ and so $W' \in SW_{\be'}$. Similarly, if $2 \in J$ and the monomial for $W$ contains $a_1$ (in which case $\be$ necessarily begins with $\si_1^{-1}$), then again $J'=J$ and $W' \in SW_{\be'}$. However, if $2 \in J$ and the monomial for $W$ does not contain $a_1$, then $1 \in J'$ and so $W' \not\in SW_{\be'}$. However, the dual walk $(W')^*$ is simple walk on $(\be')^*$ with $(J')^* \subseteq \{2,\ldots, m\}$, and hence $(W')^* \in SW_{(be')^*}$. In particular, it follows that the set $\sS_\be = SW_{\be} \cup SW_{\be^*}$ of semi-simple walks is unchanged by cyclic permutation. This completes the proof of the lemma.
\end{proof}

The next result states that, up to reordering the crossings, the set of semi-simple walks on a braid word and its rotation are equal.

\begin{proposition} 
Let $\be  = \si_{i_1}^{\ep_1} \cdots \si_{i_\ell}^{\ep_\ell}$ be a braid word on $m$ strands, and let $\be^\dag =\si_{m-i_1}^{\ep_1} \cdots \si_{m-i_\ell}^{\ep_\ell}$ be its rotation. Then the sets of semi-simple walks of $\be$ and $\be^\dag$ are equal, namely $\sS_{\be} = \sS_{\be^\dag}.$
\end{proposition}

\begin{proof}
The braid rotation $\be^\dag$ is obtained from $\be$ by rotating it 180$^\degree$ in the plane. In order to relate the semi-simple walks on $\beta$ and $\beta^\dag$, we index the crossings of $\be^\dag$ from top to bottom using $n,\ldots, 1$, and we identify semi-simple walks on $\be$ and $\be^\dag$ with the subsets of active crossings. In this way, every semi-simple walk on $\be$ and $\be^\dag$ corresponds to a monomial $a_{i_1}\ldots a_{i_k}$ indicating that $i_1,\ldots, i_k$ are the active crossings.

Given a monomial $a_{i_1}\ldots a_{i_k}$, the semi-simple walk on $\be$ is obtained by taking the smoothing of $\be$ at each crossing $i_1,\ldots, i_k$ and locally marking the active and inactive arcs using a marking scheme that differentiates them. (We mark the active arcs using zebra stripes.) Then extend the markings around the braid closure. Since $a_{i_1}\ldots a_{i_k}$ corresponds to a semi-simple walk, the markings on the active and inactive arcs will not coincide.

The same will be true for $\be^\dag$, provided one follows the same procedures at the corresponding crossings. Since $\be^\dag$ is obtained  by a 180$^\degree$ rotation which interchanges the first and last strands of the braid, this will not preserve $SW_\be$ since the new walk may not satisfy  $J \subseteq \{2,\ldots, m\}.$ Nevertheless, the semi-simple walks of $\be$ and $\be^\dag$ are preserved. This completes the proof.
\end{proof}


\subsection{Simple walks on $\mathbf{(2,n)}$ torus braids} \label{S6}  
\setcounter{section}{6} \setcounter{theorem}{0} 
In this section, we show that the number of simple walks on the braids $\be_n$ with closure the $(2,n)$ torus links are given by the Fibonacci-like sequence $1, 2, 3, 5, 8, 13, 21, 34, 55, 89, 144, \ldots$. We will find a closed form solution for the number of simple walks and use it to show they grow exponentially in $n$.

For $n\geq 1$ let $\be_n = \si_1^{-n}$. The closure of $\be_n$ is the $(2,n)$ torus knot if $n$ is odd and the $(2,n)$ torus  link if $n$ is even. For that reason, we refer to $\be_n$ as the (negative) $(2,n)$ torus braid.

\begin{proposition}
Let $f(n)$ be the number of simple walks on the $(2,n)$ torus braid $\be_n$.
Then   
\[f(n)= \left(\,\frac{5+\sqrt{5}}{10}\right)\,\left(\,\frac{1+\sqrt{5}}{2}\right)^n+\left(\,\frac{5-\sqrt{5}}{10}\right)\,\left(\,\frac{1-\sqrt{5}}{2}\right)^n.\]
\end{proposition}

\begin{proof}
The first step is to show that the simple walks on $\be_n$ satisfy the Fibonacci recurrence relation
\begin{equation} \label{recurrence}
f(n) = f(n-1) + f(n-2). 
\end{equation} 

To do this, we establish a bijective correspondence between the set of simple walks on $\be_n$ and the union of the sets of simple walks on $\be_{n-1}$ and $\be_{n-2}$. This is accomplished by extending simple walks on $\be_{n-1}$ and $\be_{n-2}$ to simple walks on $\be_n$. 

In the following, we identify simple walks with their weights, which we write as monomials in $\{a_i,b_i,c_i \mid i=1, \ldots, n \}.$ Notice that all simple walks under consideration will have $J =\{2\}.$

Given a simple walk $w'$ on $\be_{n-1}$,  set $w=w' a_n$. Then $w$ is a simple walk on $\be_n$. See \Cref{Fig-n_2} (left). Since $J = \{2\}$, this is actually the only  way  to extend $w'$ to a simple walk on $\be_n$ .

Similarly, given a simple walk $w'$ on $\be_{n-2}$, set $w =w'b_{n-1} c_n$. Then $w$ is again a simple walk on $\be_n$. See \Cref{Fig-n_2} (right). This is the only  way  to extend $w'$ to a simple walk on $\be_n$ which avoids simple walks extended from $\be_{n-1}.$

\begin{figure}[ht]
\centering
\includegraphics[scale=0.18]{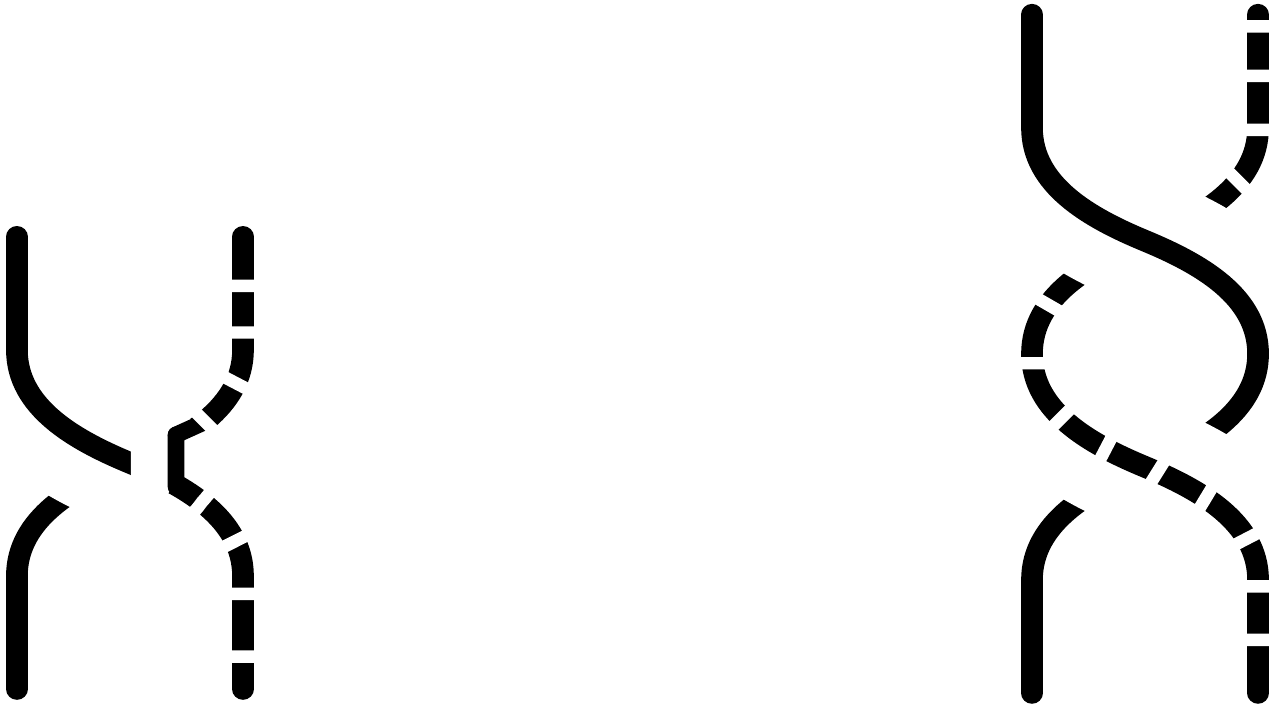} 
\caption{\small Extending simple walks from $\be_{n-1}$ and $\be_{n-2}$ to $\be_n$.}
\label{Fig-n_2}
\end{figure}

The two sets of simple walks are disjoint. This can be verified by noting that they  traverse different strands between the $(n-1)$-st and $n$-th crossings. Equivalently,  one can see this by comparing their weights at the $n$-th crossing. The simple walks extended from $\be_{n-1}$ have weight $a_n$, whereas those extended from $\be_{n-2}$ have weight $c_n$.

Every simple walk on $\be_n$ is an extension of one on $\be_{n-1}$ or $\be_{n-2}$. To that end, let $w$ be a simple walk on $\be_n$. Since $J=\{2\}$, at the $n$-th crossing, either the walk jumps down and has weight $a_n$, or it stays on the overstrand and has weight $c_n$. In the first case,  $w=w' a_n$ for a simple walk $w'$ on $\be_{n-1}.$ In the second, $w = w' b_{n-1} c_n$ for some simple walk $w'$ on $\be_{n-2}.$ This establishes the bijective correspondence, and \Cref{recurrence} follows directly.

The second step is to solve the recurrence relation \eqref{recurrence}. It is  a homogeneous linear recurrence relation with constant coefficients and characteristic polynomial
\[ p(t) = t^n - t^{n-1} - t^{n-2} = t^{n-2}(t^2-t-1).\]
This polynomial has two non-zero roots: \[ t = \frac{1 \pm \sqrt{5}}{2}. \] 
Therefore, its general solution is given by $f(n) =c_1(\frac{1 + \sqrt{5}}{2})^n + c_2(\frac{1 - \sqrt{5}}{2})^n$. Using the values $f(1)=1$ an $f(2)=2,$ it follows that
\begin{eqnarray*}
1&=&c_1\left(\frac{1 + \sqrt{5}}{2}\right) + c_2\left(\frac{1 - \sqrt{5}}{2}\right), \\
2&=&c_1\left(\frac{1 + \sqrt{5}}{2}\right)^2 + c_2\left(\frac{1 - \sqrt{5}}{2}\right)^2.
\end{eqnarray*}
Solving for $c_1,c_2$, we find that
\[ c_1 = \frac{5+\sqrt{5}}{10} \quad \text{ and } \quad c_2 =\frac{5-\sqrt{5}}{10}.\]
The formula for $f(n)$ follows, and this completes the proof.
\end{proof}


\subsection{Simple walks on $\mathbf{(3,n)}$ torus braids} \label{S7}  
\setcounter{section}{7} \setcounter{theorem}{0} 
In this section, we show that the number of simple walks on the braids $\be_n$ with closure the $(3,n)$ torus link are given by the sequence $0, 1, 4, 5, 10, 19, 34, 63, 116, 213, 392, 721, 1326,\dots$. We will find a closed form solution for the number of simple walks and use it to show they grow exponentially in $n$.

For $n\geq 1$, let $\ga_n = (\si_1^{-1} \si_2^{-1})^n$. The closure of $\ga_n$ is the $(3,n)$ torus knot if $n \not\equiv 0$ mod 3 and the $(3,n)$ torus link if $n \equiv 0$ mod 3. For that reason, we refer to $\ga_n$ as the (negative) $(3,n)$ torus braid.

\begin{proposition}
Let $g(n)$ be the number of simple walks on the $(3,n)$ torus braid $\ga_n$. Then   
\[g(n)= c_1\al^n + c_2\be^n + c_3 \ga^n,\]
where $\al,\be,\ga$ are the roots of $t^3-t^2-t-1$ (see \cref{eqn-roots} for explicit formulas for the roots) and where
 \[ c_1 = \frac{1+3\al^{-1}}{-\al^2 + 4\al -1}, \quad c_2 = \frac{1+3\be^{-1}}{-\be^2 + 4\be -1}, \quad c_3 = \frac{1+3\ga^{-1}}{-\ga^2 + 4\ga -1}.   \]
\end{proposition}

\begin{proof}
We claim  that $g(n)$ satisfies the tribonacci recurrence relation:
\begin{equation} \label{recurrence2}
 g(n) = g(n-1) + g(n-2) + g(n-3).   
 \end{equation}
The proof of the claim is long, so we first show how to solve the recurrence relation to get the formula for $g(n).$

The tribonacci numbers $T(n)$ are  the sequence $0, 1, 1, 2, 4, 7, 13, \dots$ for $n \geq 0$, and they also satisfy \eqref{recurrence2}. We will use the closed form solution for $T(n)$  to find a closed form solution for $g(n)$. The recurrence relation \eqref{recurrence2} has characteristic polynomial $p(t) = t^n - t^{n-1} -t^{n-2} - t^{n-3} = t^{n-3}(t^3-t^2-t-1).$ It has one nonzero real root $\al$ and two complex roots $\be$ and $\ga$ given by
\begin{equation}
\begin{split} \label{eqn-roots}
\al &=  \tfrac{1}{3}(1 + \sqrt[3]{19+3\sqrt{33}} + \sqrt[3]{19-3\sqrt{33}}),  \\
\be &= \tfrac{1}{2}(1 - \al + \sqrt{-3\al^2+2\al+5}),  \\
\ga &= \tfrac{1}{2}(1 - \al - \sqrt{-3\al^2+2\al+5}).   \\
\end{split}
\end{equation}
The closed form solution for $T(n)$ is a linear combination of powers of the roots of the characteristic polynomial:

\[  T(n) = \frac{\al^n}{-\al^2 + 4\al -1} + \frac{\be^n}{-\be^2 + 4\be - 1} + \frac{\ga^n}{-\ga^2 + 4\ga - 1}. \]

The sequence $g(n)$ is related to the tribonacci sequence by the equation \[ g(n) =  T(n) + 3T(n-1).   \]

From this we can write $g(n) = c_1\al^n + c_2\be^n + c_3 \ga^n$ and use the values $g(1)=0, g(2)=1, g(3)=4$ to solve for the coefficients $c_1, c_2, c_3$:
\[ c_1 = \frac{1+3\al^{-1}}{-\al^2 + 4\al -1}, \quad c_2 = \frac{1+3\be^{-1}}{-\be^2 + 4\be -1}, \quad c_3 = \frac{1+3\ga^{-1}}{-\ga^2 + 4\ga -1}.   \]

It remains to prove the claim, namely to show that the simple walks on $\ga_n$ satisfy the recurrence relation \eqref{recurrence2} for all $n \geq 4$. To do this, we establish a bijective correspondence between the simple walks on $\ga_n$ and the union  of the sets of simple walks on $\ga_{n-1}$, $\ga_{n-2}$ and $\ga_{n-3}$. In general, for braids on three strands, the simple walks will have $J = \{2\}$, $J = \{3\}$ or $J = \{2,3\}$. For the $(n,3)$ torus braids, all three occur. 

We claim that every simple walk on $\ga_{n-1}$, $\ga_{n-2}$ and $\ga_{n-3}$ can be extended to a simple walk along $\ga_n$. We prove this by considering the three cases separately. As before, we will identify simple walks with their weights, which we write as monomials in $\{a_i,b_i,c_i \mid i=1, \ldots, 2n \}.$

\begin{figure}[ht]
\centering
\includegraphics[scale=0.40]{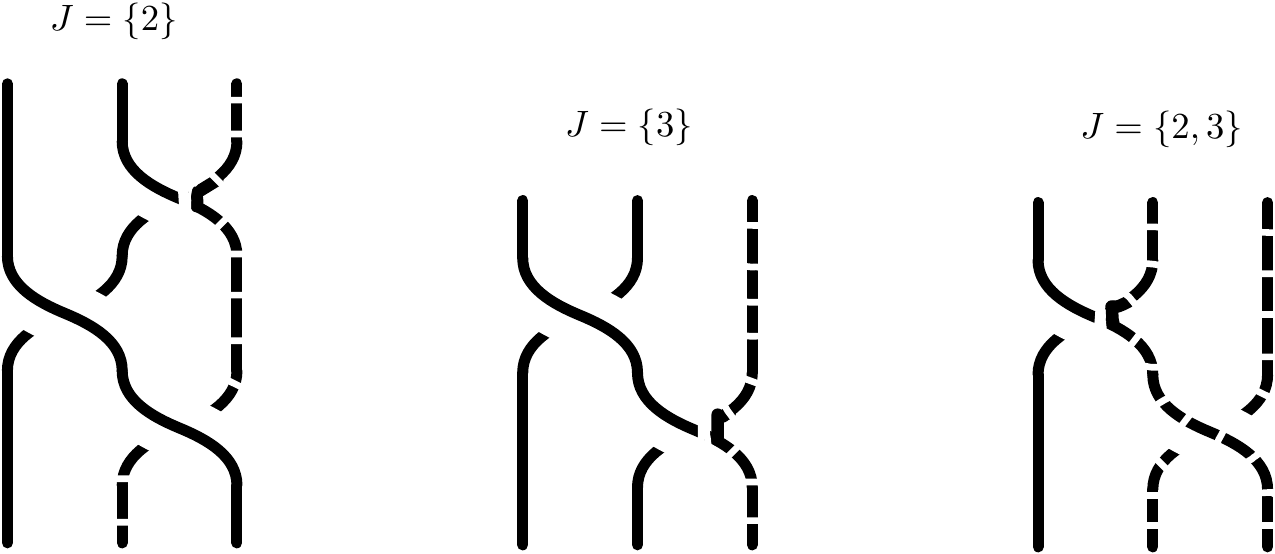} 
\caption{\small Extending simple walks from $\ga_{n-1}$ to $\ga_n$.}
\label{Fig-3n_1}
\end{figure}

Suppose $w'$ is a simple walk on $\ga_{n-1}$. If $J=\{2\}$, then it is on the understrand at the $(2n-2)$-nd crossing, and so $w'=w'' b_{2n-2}$. We set $w = w''a_{2n-2}b_{2n}$ and note that $w$ is a simple walk on $\ga_n$ with $J=\{2\}.$ If $J=\{3\}$, then we set $w=w' a_{2n}.$ If $J=\{2,3\}$, then we set $w=w'  b_{2n}\cdot a_{2n-1}c_{2n}.$  Note that in this last case, the paths become inverted, but this is allowable for the walks with $J = \{2,3\}$. \Cref{Fig-3n_1} shows how the simple walks are extended. In all three cases it is clear that $w$ is a simple walk on $\ga_n$.

\begin{figure}[hb]
\centering
\includegraphics[scale=0.40]{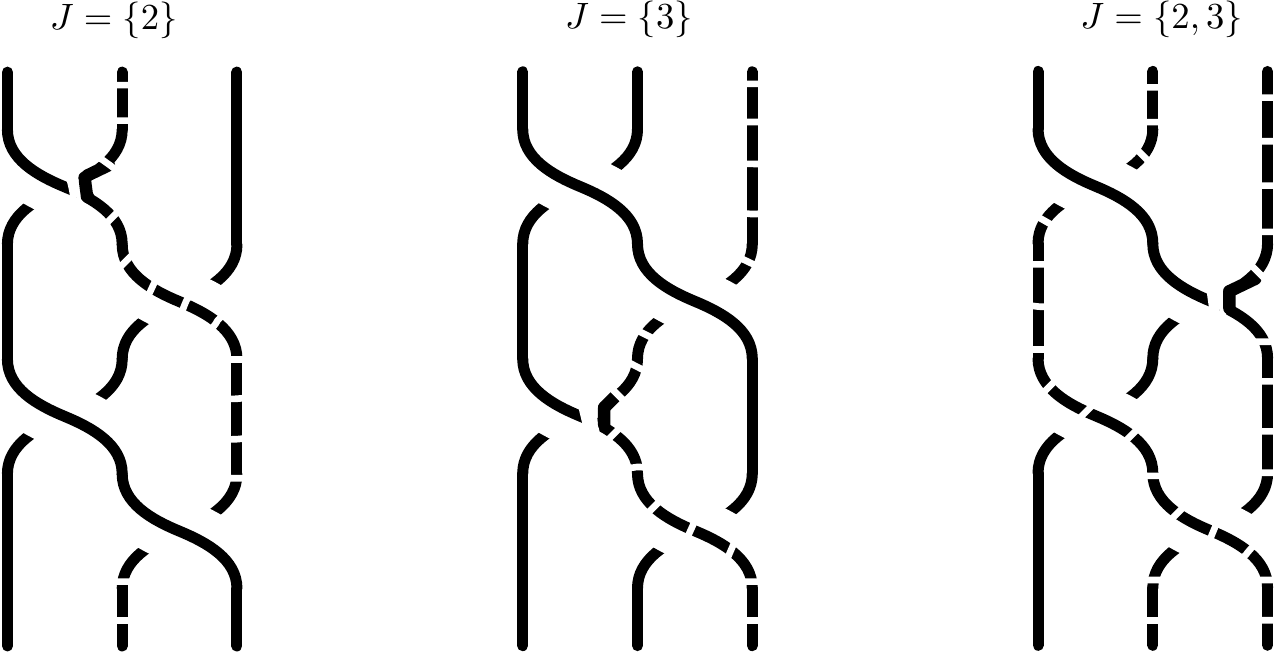} 
\caption{\small Extending simple walks from $\ga_{n-2}$ to $\ga_n$.}
\label{Fig-3n_2}
\end{figure}

In a similar way, we can extend simple walks on $\ga_{n-2}$. Let $w'$ be a simple walk on $\ga_{n-2}$. If $J=\{2\}$, then we set $w=w' a_{2n-3}c_{2n-2}b_{2n}.$ If $J=\{3\}$, then we set $w=w' b_{2n-2}a_{2n-1}c_{2n}.$ If $J=\{2,3\}$, then we set $w=w' a_{2n-2}b_{2n} \cdot b_{2n-3}c_{2n-1}c_{2n} .$  In this last case, notice that the paths become inverted. \Cref{Fig-3n_2} shows how the simple walks are extended. In all three cases, it is clear that $w$ is a simple walk on $\ga_n$.
 
Lastly, let $w'$ be a simple walk on $\ga_{n-3}.$ Then since $\ga_n = \ga_{n-3}(\si_1^{-1}\si_2^{-1})^3$ with $(\si_1^{-1}\si_2^{-1})^3$ inducing the identity permutation, we can extend $w'$ to a simple walk on $\ga_n$ by remaining on the same strands. If $J=\{2\}$, then we set $w=w' b_{2n-5}c_{2n-3} c_{2n-2}b_{2n}.$ If $J=\{3\}$, then we set $w=w' b_{2n-4}b_{2n-3}c_{2n-1}c_{2n}.$ If $J=\{2,3\}$, then we set $w=w' b_{2n-5}c_{2n-3}c_{2n-2}b_{2n}  \cdot b_{2n-4}b_{2n-3}c_{2n-1}c_{2n}.$ \Cref{Fig-3n_3} shows how the simple walks are extended. In all three cases, it is clear that $w$ is a simple walk on $\ga_n$. It is clear that $w$ is a simple walk on $\ga_n$.

\begin{figure}[ht]
\centering
\includegraphics[scale=0.40]{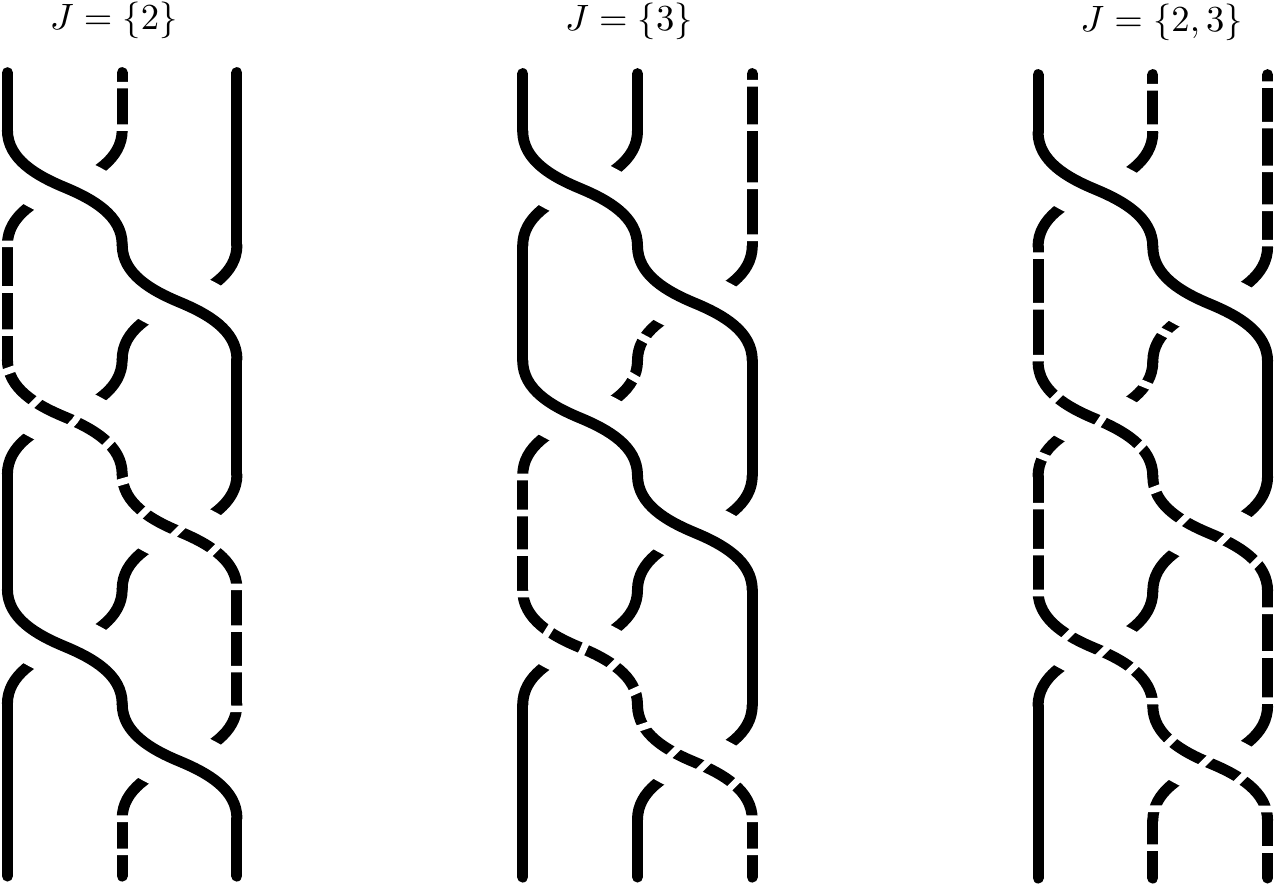} 
\caption{\small Extending simple walks from $\ga_{n-3}$ to $\ga_n$.}
\label{Fig-3n_3}
\end{figure}

It is not difficult to see that the sets of extended simple walks from $\ga_{n-1},\ga_{n-2},$ and $\ga_{n-3}$ are all disjoint. For instance, this follows by comparing their weights, and noting they are pairwise unequal. 

The last step is to show that this accounts for all simple walks on $\ga_n$. We will see that  every simple walk on $\ga_n$ is an extension of a simple walk on $\ga_{n-1}$, $\ga_{n-2}$ or $\ga_{n-3}$.

Suppose $w$ is a simple walk on $\ga_n$.  If $J = \{2\}$, then it must traverse the understrand at the $2n$-th crossing and is on the overstrand just below the $(2n-2)$-nd crossing. If it jumps down at the $(2n-2)$-nd crossing, then $w = w'' a_{2n-2} b_{2n}$, and $w$ is the extension of the simple walk $w' = w'' n_{2n-2}$ on $\ga_{n-1}$.  Otherwise, if it remains on the overstrand at the $(2n-2)$-nd crossing, then either  $w = w' a_{2n-3} c_{2n-2} b_{2n}$ for $w'$ a simple walk on $\ga_{n-2}$, or $w = w' b_{2n-5} c_{2n-3} c_{2n-2} b_{2n} $ for $w'$ a simple walk on $\ga_{n-3}$.

If instead $J = \{3\}$, then it is on the overstrand just below the $(2n)$-th crossing.  If it jumps down at the $(2n)$-th crossing, then $w= w' a_{2n}$ for $w'$ a simple walk on $\ga_{n-1}.$ Otherwise, if it remains on the overstrand at the $(2n)$-th crossing, then either $w= w' b_{2n-2}a_{2n-1}c_{2n}$ for $w'$ a simple walk on $\ga_{n-2}$, or $w= w'b_{2n-4}b_{2n-3}c_{2n-1}c_{2n}$ for $w'$ a simple walk on $\ga_{n-3}$.

Lastly, suppose $J = \{2,3\}$. Then $w$ consists of two paths, ends with $b_{2n}c_{2n},$ and is on the overstrand below the $(2n-1)$-st crossing. If it jumps down at the $(2n-1)$-st crossing, then $w=w' b_{2n} \cdot a_{2n-1}c_{2n}$ for $w'$ a simple walk on $\ga_{n-1}$. If it remains on the overstrand at the $(2n-1)$-st crossing, then it is also on the overstrand at the $(2n-2)$-nd crossing. If it jumps down at the $(2n-2)$-nd crossing, then $w=w'  a_{2n-2}b_{2n} \cdot b_{2n-3}c_{2n-1}c_{2n}$ for $w'$ a simple walk on $\ga_{n-2}$. If it remains on the overstrand at the $(2n-2)$-nd crossing, then it approaches the $(2n-4)$-th and $(2n-5)$-th crossings on understrands, and it follows that $w=w' b_{2n-5} c_{2n-3}c_{2n-2}b_{2n} \cdot  b_{2n-4} b_{2n-3} c_{2n-1} c_{2n}$ for $w'$ a simple walk on $\ga_{n-3}$.

This shows that every simple walk on $\ga_n$ is obtained by extending a simple walk on $\ga_{n-1}, \ga_{n-2}$ or $\ga_{n-3}.$ It follows that there is a bijection correspondence between the set of simple walks on $\ga_n$ and the union of the simple walks on $\ga_{n-1}, \ga_{n-2}$ and $\ga_{n-3}.$ The bijective correspondence implies that the sequence $g(n)$ of simple walks on $\ga_n$ satisfy the recurrence relation \eqref{recurrence2}.
\end{proof}

\subsection{Minimal braid representatives} \label{S8}  
\setcounter{section}{8} \setcounter{theorem}{0} 
In \Cref{table5}, we list  knots up to 9 crossings with the braid representatives giving minimal numbers of simple walks.  More extensive tables of knots up to 13 crossings and braid representatives for them can be found online at \cite{BS-ancillary}.
(In \cite{BS-ancillary} and \Cref{table1} below, we use the notation for braid words from sagemath, meaning that a braid word $ \si_{a_1}^{\ep_1}\cdots \si_{a_\ell}^{\ep_\ell}$ is denoted by $[\ep_1 a_1,\ldots, \ep_\ell a_\ell]$. Tables \ref{table4} and \ref{table5} use even more compactified notation similar to that at the end of \cite{Jones-1987}.)

These results are empirical. The braid words listed in \Cref{table5} and \cite{BS-ancillary} are the output of a sagemath program developed by the second author. It takes as input braid representatives for knots (given by the braids from  \cite{knotinfo} and \cite{Stoimenow-2002}) and applies cyclic permutation, reflection, and rotation. It then selects the braid word that minimizes the number of simple walks. The output braid word may represent the knot $K$ or its mirror image $K^*$, whichever has fewest simple walks. 

The braids listed have the fewest simple walks among all \textit{known} braid representatives for the given knots. In general, the question of finding a complete list of braid representatives for a given knot is a delicate and open problem. As we shall see, it is not enough to consider only braid representatives of minimal width. Even if it were, it is an open problem to develop an algorithm for computing the braid width of a knot (see Open Problem 1 in \cite{Birman-Brendle-2005}). Nevertheless, these problems have been studied extensively, and much is known about minimal braid representatives of low-crossing knots; see \cite{Jones-1987}, \cite{Gittings}, and \cite{Stoimenow-2002}.

Given a knot, one can look for braid representatives that minimize its braid width or the braid length. For many knots, there is a braid representative that simultaneously minimizes both the width and length, but in general, the braid representatives that minimize width need not be the same as the ones that minimize length. The earliest known examples are the knots $16_{472381}$ and $16_{1223549}$, which were discovered by Stoimenow and have braid width 4 but no minimal length braid representative of width 4, \cite[Figure 7]{Stoimenow-2002}. 

This interesting aspect has been further studied by Gittings \cite{Gittings} and Van Cott \cite{VanCott}, and the ``smallest'' example is the knot $10_{136}$. For all other knots with up to 10 crossings,  there is a braid representative that simultaneously minimizes the braid width and length. Further examples of knots whose minimal width braid representatives are not minimal length are listed in \Cref{table4}. (These examples come from \cite{knotinfo}.) 

Interestingly, the braid representative that minimizes the number of simple walks is not always a minimal length braid, nor is it always a minimal width braid either. For example, consider the knots $10_{136}$ and $11n_8$ and their braid representatives in \Cref{table4}. For $10_{136}$, the number of simple walks is minimized on a braid representative of minimal length but not one of minimal width, whereas for $11n_{8}$, the number of simple walks is minimized on a braid representative of minimal width, but not one of minimal length.
Similar examples can be found among the other knots listed in \Cref{table4}.

Our computations suggest that, for any knot, one can always minimize the number of simple walks on a braid representative of minimal width or minimal length. This is an interesting problem for future investigation. In order to make progress, we need more information about the minimal width and minimal length braid representatives for knots. At present, we do not have complete information on the 13-crossing knots. In particular, we do not know which 13-crossing knots have minimal length braid representatives that are not minimal width. The braid representatives for the 13-crossing knots from [Sto02] are known to be of minimal width, but they are not known to be of minimal length.

Notice that for every knot in \Cref{table5}, the braid representative that minimizes the number of simple walks begins with $\si_1$. This is actually true for all knots up to 10 crossings,  but not immediately true for knots with 11 or more crossings (see \cite{BS-ancillary}).

\begin{table}
{\footnotesize 
\begin{tabular}{|c|l|c|}  \hline
{\bf Knot} & \qquad {\bf Braid Word} & {\bf |SW$_{\be}$| } \\ \hline  
$\mathbf{11a_{322}}$ & $[-1,-1,2,-3,4,-3,2,-3,4,1,2,-3,-2,-2]$  & 51 \\
$\mathbf{12a_{23}}$ & $[-1,-3,2,-3,-5,2,4,1,2,-3,-4,5,4,-3,-5,4,2]$ & 153 \\
$\mathbf{12a_{155}}$ & $[-1,2,2,-3,4,-3,4,-5,-4,3,2,1,-4,5,2,-3,2]$ & 127 \\
$\mathbf{12a_{288}}$ & $[-1,2,2,-3,2,-3,2,1,2,-3,4,2,-3,4]$ & 71 \\
$\mathbf{12a_{449}}$ & $[-1,2,4,-3,4,5,4,2,-3,-4,-4,-5,1,2,4,-3,2]$ & 125 \\
$\mathbf{12a_{494}}$ & $[-1,2,-3,4,5,2,-3,-4,-4,1,2,-3,4,-5,4,-3,2]$ & 137 \\
$\mathbf{12a_{750}}$ & $[-1,2,-3,2,-3,5,4,-3,-4,-4,-5,-4,-4,-3,1,4,-2,3,-2]$ & 183 \\
$\mathbf{12n_{546}}$ & $[-1,2,3,1,-2,1,1,1,1,-2,-3,-3,2]$ & 41 \\
$\mathbf{12n_{601}}$ & $[-1,-1,2,3,3,3,2,-4,-4,-3,1,2,3,3,-4,2]$ & 67 \\
$\mathbf{12n_{622}}$ & $ [-1,-1,2,2,3,3,2,-4,1,-2,3,-4,-2,3]$ & 47 \\ \hline
\end{tabular}
\vspace{2mm}
\caption{\small Knots up to 12 crossings whose minimizing braid word begins with $\si_1^{-1}$.} \label{table1}
}
\end{table}

In general, by \Cref{lemma-1}, the minimizing braid representative can always be chosen to begin with either $\si_1$ or $\si_1^{-1}.$   For knots with 11 and 12 crossings, there are only a handful of examples whose minimizing braid representative begins with $\si_1^{-1}$ and not $\si_1$. They are listed in \Cref{table1}. (There are in addition 82 examples among the knots with 13 crossings, see \cite{BS-ancillary}.)  In each case, we can find a minimizing braid representative that begins with $\si_1$ by reversing the braid word and applying cyclic permutation. We explain these steps in more detail.

Take, for example, the first knot in \Cref{table1}, namely $11a_{322}$.  Its minimizing braid representative is $\si_1^{-2}\si_2 \si_3^{-1} \si_4 \si_3^{-1} \si_2 \si_3^{-1}\si_4 \si_1 \si_2\si_3^{-1}\si_2^{-2}.$  The reversed braid word will have the same number of simple walks, so  $\si_2^{-2}\si_3 \si_2 \si_1 \si_4 \si_3^{-1} \si_2 \si_3^{-1} \si_4 \si_3^{-1}\si_2 \si_1^{-2}$ is also a minimizing braid word for $11a_{322}$. Now repeated application of \Cref{lemma-1} shows that $\si_1 \si_4 \si_3^{-1} \si_2 \si_3^{-1} \si_4 \si_3^{-1}\si_2 \si_1^{-2} \si_2^{-2}\si_3\si_2$ is also a minimizing braid word for $11a_{322}$.

The same method applies to the other knots in \Cref{table1}. Each one admits a minimizing braid word that starts with $\si_1$. A similar argument applies to the braid representatives for the 13 crossing knots that begin with $\si_1^{-1}$. This follows by a routine but somewhat tedious exercise. 

\Cref{table2} shows the growth rate of the number of simple walks as a function of the crossing number of the knot. \Cref{table3} shows the growth rate of the number of simple walks as a function of the braid length.  Note that \Cref{table3} contains information for knots with up to 12 crossings but not the 13-crossing knots. The reason is that we do not have definitive information about the braid representatives of minimal length for the 13-crossing knots.

\pgfplotstableread[row sep=\\,col sep=&]{
    interval & SW  \\
    3     & 1.0    \\
    4     & 2.0   \\
    5     & 2.5   \\
    6     & 4.0   \\
    7     & 6.142857    \\
    8     & 8.190476    \\
    9     & 13.183574   \\
    10     & 17.163636    \\
    11     & 26.461957    \\
    12     & 40.7284    \\ 
    13     & 54.9324189    \\ 
    }\SWdata

\begin{table}
\begin{tikzpicture}
    \begin{axis}[
            width  = 0.75*\textwidth,
            height = 6cm,
            ybar,
            ymin=0,ymax=70,
            bar width=.52cm,
            symbolic x coords={3,4,5,6,7,8,9,10,11,12,13},
            xtick=data,
        ]
        \addplot[draw=black, semithick, fill= gray] table[x=interval,y=SW]{\SWdata};
        \legend{\small Average Number of Simple Walks}
    \end{axis}
\end{tikzpicture}

\caption{\small Average number of simple walks by crossing number.} \label{table2}
\end{table}

\pgfplotstableread[row sep=\\,col sep=&]{
    interval & SW  \\
    3     & 1.0    \\
    4     & 2.0   \\
    5     & 3.0   \\
    6     & 3.667   \\
    7     & 6.75    \\
    8     & 8.142857   \\
    9     & 11.71875   \\
    10     & 16.634920    \\
    11     & 21.093167   \\
    12     & 32.180438    \\
    13     & 32.619109    \\
    14	   & 40.130711 \\
    15	  & 53.921212\\
    16	& 58.711656\\
    17	& 69.3442623\\ 
        }\SWdata

\begin{table}
\begin{tikzpicture}
    \begin{axis}[
            width  = 0.85*\textwidth,
            height = 6cm,
            ybar,
            ymin=0,ymax=110,
            bar width=.52cm,
            symbolic x coords={3,4,5,6,7,8,9,10,11,12,13,14,15,16,17},
            xtick=data,
        ]
        \addplot[draw=black, semithick, fill= gray] table[x=interval,y=SW]{\SWdata};
        \legend{\small Average Number of Simple Walks}
    \end{axis}
\end{tikzpicture}
\caption{\small Average number of simple walks by braid length.} \label{table3}
\end{table}

We end this paper with a few questions and open problems for future investigation. One is whether braid words that minimize the number of simple walks have a preferred shape or form. By \Cref{prop:reduced-irreducible}, we can assume the braid word is reduced and irreducible, and by \Cref{lemma-1}, we can assume it begins with $\si_1$ or $\si_1^{-1}.$ We conjecture the braid word can always be chosen to begin with $\si_1$. Further results as to the shape of minimizing braid words would be helpful for developing efficient search algorithms.

More generally, it would be extremely useful to automate the generation of minimal width and/or minimal length braid representatives for a given knot. Such tools would allow fast computation of the colored Jones polynomial and other quantum knot invariants, enabling calculations for higher crossing knots, including those in the knot tables of Burton \cite{Burton-2020}, who has recently  extended the classification of knots to 19 crossings.

\subsection*{Auxilliary files} Sagemath programs and datasets are available online \cite{BS-ancillary}. This includes a program that generates all the simple walks as operators for a given braid
and another that selects braid words that minimize the number of simple walks. It also includes input datasets used to create Tables \ref{table4} and \ref{table5}, as well as output datasets of braid words that minimize the number of simple walks for knots up to 13 crossings.

\subsection*{Acknowledgements}
The first author would like to acknowledge funding from the Natural Sciences and Engineering Research Council of Canada, and the second author acknowledges funding from a USRA award and a Stewart award from McMaster University.
The authors are especially grateful to Alexander Stoimenow for providing crucial input. They 
would also like to thank Homayun Karimi, Robert Osburn, Andrew Nicas, Will Rushworth, and Cornelia Van Cott for valuable feedback.  

\renewcommand{\arraystretch}{1.15}
\begin{table}
{\footnotesize 
{\rowcolors{1}{white!80!white!50}{lightgray!70!lightgray!40}
\begin{tabular}{|c|c|c|c|c|}  \hline
{\bf Knot} & {\bf Minimal width braid} & \; {\bf |SW$_\be$|} \; &  {\bf Minimal length braid} & \; {\bf |SW$_\be$|}\; \\ \hline  
$\mathbf{10_{136}}$ & $\mathbf{1 2 3^{-1} 2 1^{-1} 2^2 3^{-1} 2^{-2} 1}$ & 21 &
$\mathbf{1 2^{-1} 3^{-1} 2^2 4 3^{-1} 4 1 2^{-1}}$ & 17 \\ 
$\mathbf{11n_8}$ & $\mathbf{123^{-1}2^{-1} 1 2^{-1} 1^{-1} 2 3^2 2^2 1}$ & 20
& $\mathbf{1 2 1^{-1} 2 1 3^{-1} 2^2 4 3^{-1} 4 1} $ & 23  \\  
$\mathbf{11n_{121}}$ & $\mathbf{1 2 3^2 2^{-1} 1 2^{-1}1^{-1} 2^2 3^{-1} 2 1} $ & 20 & 
$\mathbf{1 2 1^2 3^{-1} 2 1^{-1} 2 4 3^{-1} 4 1}$ & 21 \\ 
$\mathbf{11n_{131}}$ & $\mathbf{1 2 3^2  2^{-2} 1^2 2^{-1} 3^{-1} 2 1 2^{-1}} $ & 17 & 
$\mathbf{12 3^{-1} 2 1 3^2 2^{-1} 3^{-1} 4 3^{-1} 4}$ & 21 \\  
$\mathbf{12n_{17}}$ & $\mathbf{2^{-1} 1 2^{-1} 3 4^{-2} 3 2^{-2} 3^{-1} 4 3^{-1} 4 1} $ & 63 & 
$\mathbf{(1 2^{-1})^2  3^{-1} 2^2 4 3^2 5^{-1} 4 5^{-1} }$ & 52 \\  
$\mathbf{12n_{20}}$ & $\mathbf{3^{-1} 1 2^{-1} (3 2^{-1})^2 1^{-1} 2 3^{-1} 2^2 1} $ & 26 & 
$\mathbf{(1 2^{-1})^2 (3^{-1} 2)^2 4 3^{-1} 4 1}$ & 32 \\ 
$\mathbf{12n_{24}}$ & $\mathbf{1 2^{-1} 3 2^{-2} 3 2^{-1} 3^{-1} 1^{-1} 2 3^{-1} 2 1} $ & 35 & 
$\mathbf{2^{-1} 1 2^{-1} 3^{-1} 2^2 4 3^{-3} 4 1}$ & 32 \\ 
$\mathbf{12n_{65}}$ & $\mathbf{2^{-1} 1 2 3^{-1} 4 2 4 3^{-1} 2 3^{-1} 4^{-1} 2 3^{-1} 1} $ & 26 &
$\mathbf{1 2 1^{-1} 3^{-1} 2 3^{-1} 4^{-1} 3^2 5 4^{-1} 5 1}$ & 25 \\ 
$\mathbf{12n_{119}}$ & $\mathbf{2^{-1} 1^2 2^{-1} 3 1 3 (1 2^{-1})^2 3^{-1} 1} $ & 31 &
$\mathbf{2^{-1} 1 3^{-1} 2 4 3^{-1} 4 1^2 2^{-1} 1^2}$ & 23 \\ 
$\mathbf{12n_{284}}$ & $\mathbf{(12^{-2})^2 3 1 3 2^{-1} 1 2^{-1}3^{-1}} $ & 36 &
$\mathbf{2^{-1} 1 2^{-1} 3^{-1} 2 3^{-1} 4 3^{-1} 2 3^{-1} 4 1}$ & 32 \\ 
$\mathbf{12n_{311}}$ & $\mathbf{412 1^{-1} 2 3^{-1} 4 2^{-1} 3^{-1} 4 3 2^{-1} 3^{-1} 1} $ & 37 &
$\mathbf{2^{-1} 1 3 2 4^{-1} 3^{-1} 2 3^{-1} 5 4^{-1} 3 5 1}$ & 30 \\ 
$\mathbf{12n_{314}}$ & $\mathbf{2^{-1} 3^{-1} (1 2^{-1})^3 3 1 3 2^{-1} 1} $ & 34 &
$\mathbf{(1 2^{-1})^2 3^{-1} 2 1 4 3^{-1} 2^{-1} 4 3}$ & 24 \\ 
$\mathbf{12n_{358}}$ & $\mathbf{2^{-1} 3^{-1} 1 2 1^{-1} 2 1 2^{-1} 3 1 3 2^{-1} 1} $ & 24 &
$\mathbf{3^{-1} 2^{-1} 1 4 3 2^{-1} 3 4^{-1} 3 4 2 1}$ & 20 \\ 
$\mathbf{12n_{362}}$ & $\mathbf{1 2 3^{-2} 2 1^{-1} 2^2 3^{-1} 2^{-1} 3 2^{-1} 1} $ & 42 &
$\mathbf{3^{-1} 2^{-1} 1 2^{-2} 4 3^{-1} 4 1^2 2 1}$ & 24 \\ 
$\mathbf{12n_{403}}$ & $\mathbf{1 2^{-1} 3^4 2^{-1} 3^{-1} 1^{-1} 2 3^{-1} 2 1} $ & 37 &
$\mathbf{1 2 1 3^{-1} 2^{-1} 1 2^{-1} 4 3^{-1} 2 4 1}$ & 20 \\ 
$\mathbf{12n_{482}}$ & $\mathbf{2^{-1} 3^{-1} 1 3^{-1} 2^2 1^{-2} 2 3 2^{-1} 1^2} $ & 29 &
$\mathbf{2^{-1} 1 2^{-1} 3^{-1} 2 4 3^{-1} 2 3^{-2} 4 1}$ & 32 \\ \hline
\end{tabular}
}
\vspace{2mm}
\caption{\small Simple walks for representatives of minimal braid width and length.} \label{table4}
}
\end{table}
\renewcommand{\arraystretch}{1.00}

\newpage
\newcommand{\etalchar}[1]{$^{#1}$}

\begin{table}
\renewcommand{\arraystretch}{1.15}
\begin{tabular}{lr}
{\footnotesize \begin{minipage}{0.5\textwidth}
{\rowcolors{1}{white!80!white!50}{lightgray!70!lightgray!40}
\begin{tabular}{|c|c|c|}  \hline
{\bf Knot} & {\bf Braid Word} & {\bf |SW$_{\be}$| } \\ \hline  
$\mathbf{3_1}$ & $\mathbf{1^3} $ & 1 \\ 
$\mathbf{4_1}$ & $\mathbf{(1 2^{-1})^2} $ & 2 \\ 
$\mathbf{5_1}$ & $\mathbf{1^5} $ & 3  \\ 
$\mathbf{5_2}$ & $\mathbf{12^3 12^{-1}} $ & 2 \\  
$\mathbf{6_1}$ & $\mathbf{121^{-1}3^{-1}23^{-1}1} $ & 3 \\ 
$\mathbf{6_2}$ & $\mathbf{1^32^{-1}12^{-1}} $ & 4\\ 
$\mathbf{6_3}$ & $\mathbf{1^2 2^{-1}12^{-2}} $ & 5\\ 
$\mathbf{7_1}$ & $\mathbf{1^7} $ & 8  \\ 
$\mathbf{7_2}$ & $\mathbf{123^3 1^{-1}213^{-1}} $ & 3 \\ 
$\mathbf{7_3}$ & $\mathbf{1^4 2 1^{-1} 21} $ & 6  \\ 
$\mathbf{7_4}$ & $\mathbf{12^232^{-1}321 2^{-1} } $ & 5\\ 
$\mathbf{7_5}$ & $\mathbf{1^3 21^{-1}2^21} $ & 5 \\ 
$\mathbf{7_6}$ & $\mathbf{1^2 2^{-1} 132^{-1}3} $ & 8 \\ 
$\mathbf{7_7}$ & $\mathbf{(12^{-1})^2 32^{-1}3} $  & 8 \\  
$\mathbf{8_1}$ & $\mathbf{1234^{-1}2^{-1}3212^{-1}4^{-1}} $ & 4 \\ 
$\mathbf{8_2}$ & $\mathbf{1^5 2^{-1}12^{-1}} $ & 9 \\
$\mathbf{8_3}$ & $\mathbf{121^{-1}3^{-1}23^{-1}4^{-1}34^{-1} 1} $ & 9 \\ 
$\mathbf{8_4}$ & $\mathbf{123^{-1}23^{-3} 12^{-1}} $ & 7 \\ 
$\mathbf{8_5}$ & $\mathbf{(1^3 2^{-1})^2 } $ & 8 \\ 
$\mathbf{8_6}$ & $\mathbf{1^3 2 1^{-1} 3^{-1} 2 3^{-1} 1} $ & 7 \\ 
$\mathbf{8_7}$ & $\mathbf{1^4 2^{-1} 1 2^{-2} } $ & 10 \\ 
$\mathbf{8_8}$ & $\mathbf{1^2 2 1^{-1} 3^{-1} 2 3^{-2} 1 } $ & 10 \\ 
$\mathbf{8_9}$ & $\mathbf{1^3 2^{-1} 1 2^{-3} } $ & 9 \\ 
$\mathbf{8_{10}}$ & $\mathbf{1^3 2^{-1} 1^2 2^{-2} } $ & 9 \\ 
$\mathbf{8_{11}}$ & $\mathbf{1 2 1^{-1} 2^2 3^{-1} 2 3^{-1} 1 } $ & 7 \\ 
$\mathbf{8_{12}}$ & $\mathbf{( 1 2^{-1}34^{-1})^2 } $ & 14 \\ 
$\mathbf{8_{13}}$ & $\mathbf{1 2^2 3^{-1} 2 3^{-2} 1 2^{-1}} $ & 8 \\ 
$\mathbf{8_{14}}$ & $\mathbf{1^2 2 1^{-1} (2 3^{-1})^2 1 } $ & 8 \\ 
$\mathbf{8_{15}}$ & $\mathbf{1 2^3 1 3 2^{-1} 3^2 } $ & 9 \\ 
$\mathbf{8_{16}}$ & $\mathbf{1^2 2^{-1} 1 2^{-1} 1^2 2^{-1} } $ & 9 \\ 
$\mathbf{8_{17}}$ & $\mathbf{1 (1 2^{-1})^3 2^{-1} } $ & 9 \\ 
$\mathbf{8_{18}}$ & $\mathbf{(1 2^{-1})^4 } $ & 10 \\ 
$\mathbf{8_{19}}$ & $\mathbf{(1 2^3)^2 } $ & 5 \\ 
$\mathbf{8_{20}}$ & $\mathbf{1 2^3 1 2^{-3} } $ & 5 \\ 
$\mathbf{8_{21}}$ & $\mathbf{1^2 2 1^{-2} 2^2 1 } $ & 6 \\ 
$\mathbf{9_1}$ & $\mathbf{1^9} $ & 21   \\
$\mathbf{9_2}$ & $\mathbf{ 12343^{-1} 42^{-1} 3^2 21 2^{-1} } $ & 5     \\ 
$\mathbf{9_3}$ & $\mathbf{ 1^6 21^{-1} 21 } $ & 14    \\ 
$\mathbf{9_4}$ & $\mathbf{ 123^2 2^4 3^{-1} 12 } $ & 9    \\ 
$\mathbf{9_5}$ & $\mathbf{ 12 1^{-1} 2342^{-1}3^{-1} 4321 } $ & 10    \\ 
$\mathbf{9_6}$ & $\mathbf{ 1^5 2 1^{-1} 2^2 1 } $ & 12     \\ 
$\mathbf{9_7}$ & $\mathbf{ 12 3^4 2^2 3^{-1} 12^{-1}  } $ & 9    \\ \hline
\end{tabular}
}
\end{minipage} 

\begin{minipage}{0.5\textwidth}
{\rowcolors{1}{white!80!white!50}{lightgray!70!lightgray!40}
\begin{tabular}{|c|c|c|}  \hline
{\bf Knot} & {\bf Braid Word} & {\bf |SW$_\be$| } \\ \hline  
$\mathbf{9_8}$ & $\mathbf{1412^{-1}3^{-1}423^{-1}1 2^{-1}} $ &  13 \\  
$\mathbf{9_9}$ & $\mathbf{ 1 2^3 1^{-1}2 1^4 } $ &  12 \\ 
$\mathbf{9_{10}}$ & $\mathbf{ 123 2^{-1} 3 2^3 1^{-1} 21 } $ &  11 \\  
$\mathbf{9_{11}}$ & $\mathbf{ 12^{-1} 1^3 312^{-1}3 } $ &  15 \\  
$\mathbf{9_{12}}$ & $\mathbf{ 123^{-1} 4 1^{-1} 23^{-1} 414} $ &  13 \\  
$\mathbf{9_{13}}$ & $\mathbf{ 123 2^{-1} 321^{-1} 2^3 1 } $ &  10 \\  
$\mathbf{9_{14}}$ & $\mathbf{ 121^{-1} 3^{-1} 23^{-1} 43^{-1} 41 } $  &  11 \\  
$\mathbf{9_{15}}$ & $\mathbf{ 1^2 23^{-1} 41^{-1} 23^{-1} 41 } $ &  16 \\  
$\mathbf{9_{16}}$ & $\mathbf{ 12^3 1^{-1}2^2 1^3 } $ &  10 \\  
$\mathbf{9_{17}}$ & $\mathbf{ (12^{-1})^2 2^{-1} (32^{-1})^2 } $ &  17 \\ 
$\mathbf{9_{18}}$ & $\mathbf{ 1 2^2 3^2 2^3 12^{-1} 3^{-1} } $ &  10 \\ 
$\mathbf{9_{19}}$ & $\mathbf{ 12^{-2} 3^{-1} 243^{-1} 41 2^{-1} } $  &  14 \\ 
$\mathbf{9_{20}}$ & $\mathbf{ 1^3 2^{-1} 3132^{-1} 3 } $ &  15 \\  
$\mathbf{9_{21}}$ & $\mathbf{ 121^{-1} 23^{-1} 243^{-1} 41 } $ &  12 \\  
$\mathbf{9_{22}}$ & $\mathbf{ 12^{-1} 31 2^{-3} 32^{-1} } $ &  17 \\ 
$\mathbf{9_{23}}$ & $\mathbf{ 12^2 1^{-1} 23^{-1} 2^2 13^2 } $  &  15 \\  
$\mathbf{9_{24}}$ & $\mathbf{ 131 2^{-1} 31 2^{-3} } $ &  17 \\  
$\mathbf{9_{25}}$ & $\mathbf{ 1 2^3 34^{-1} 12^{-1} 34^{-1} } $ &  14 \\  
$\mathbf{9_{26}}$ & $\mathbf{ 12^{-1}1^2 312^{-1}3  2^{-1}} $ &  13 \\ 
$\mathbf{9_{27}}$ & $\mathbf{ 1^2 2^{-1}12^{-2} 32^{-1} 3 } $  &  15 \\  
$\mathbf{9_{28}}$ & $\mathbf{ 1 2^{-1} 131 2^{-2} 3^2 } $ &  17 \\ 
$\mathbf{9_{29}}$ & $\mathbf{ (12^{-1}32^{-1})^2 2^{-1} } $ &  16 \\  
$\mathbf{9_{30}}$ & $\mathbf{ 1^2 2^{-2} 12^{-1} 32^{-1} 3  } $  &  16 \\ 
$\mathbf{9_{31}}$ & $\mathbf{ 12^{-1} 1312^{-1} 3^2 2^{-1} } $ &  15 \\  
$\mathbf{9_{32}}$ & $\mathbf{ 1 (1 2^{-1})^2 312^{-1}3 } $ &  14 \\ 
$\mathbf{9_{33}}$ & $\mathbf{ (12^{-1})^2 2^{-1} 312^{-1} 3 } $  &  16 \\  
$\mathbf{9_{34}}$ & $\mathbf{ 12^{-1} 3 (12^{-1})^2 32^{-1} } $  &  13 \\  
$\mathbf{9_{35}}$ & $\mathbf{ 1234^{-1} 341^{-1} 42^{-1} 321 23^{-1} } $ &  17 \\  
$\mathbf{9_{36}}$ & $\mathbf{ 1^3 2^{-1} 1312^{-1}3 } $ &  14 \\  
$\mathbf{9_{37}}$ & $\mathbf{ (12^{-1} 3)^2 43^{-1} 23^{-1} 2^{-1}4^{-1}} $ &  29 \\  
$\mathbf{9_{38}}$ & $\mathbf{ 123^2 2 1^{-1} 23^{-1} 2^2 1 } $ &  13 \\ 
$\mathbf{9_{39}}$ & $\mathbf{ 123^{-1}2143^{-1}2^{-1}4^{-1}34^2 } $ &  18 \\  
$\mathbf{9_{40}}$ & $\mathbf{ 12^{-1} 312^{-1} 132^{-1} 3 } $ &  14 \\ 
$\mathbf{9_{41}}$ & $\mathbf{ 1 3^{-1} 423^{-1} 23^{-2} 1 2^{-1} 34} $  &  20 \\  
$\mathbf{9_{42}}$ & $\mathbf{ 123^{-1} 21 2^{-3} 3^{-1}} $ &  7 \\ 
$\mathbf{9_{43}}$ & $\mathbf{ 1 2^3 3^{-1} 123^{-1} 2 } $  &  8 \\  
$\mathbf{9_{44}}$ & $\mathbf{ 123^{-1} 2^2 12^{-1} 3^{-1}2^{-1} } $ &  7 \\  
$\mathbf{9_{45}}$ & $\mathbf{ 1 2^3 312^{-1} 3 2^{-1}} $ &  8 \\  
$\mathbf{9_{46}}$ & $\mathbf{ 123^{-1} 212^{-1} 32^{-1}3 } $ &  8 \\ 
$\mathbf{9_{47}}$ & $\mathbf{ (123^{-1} 2)^2 3^{-1} } $ &  8 \\  
$\mathbf{9_{48}}$ & $\mathbf{ 123^2 2^{-1} 12^{-1} 3^{-1}1^{-1}21 } $ &  9 \\ 
$\mathbf{9_{49}}$ & $\mathbf{ 12^2 131 2^{-1} 321 2^{-1}} $ &  9 \\  \hline
\end{tabular}
}
\end{minipage} 
}
\end{tabular}
\renewcommand{\arraystretch}{1.00}

\vspace{2mm}
\caption{\small Knots up to 9 crossings and braid word representatives that minimize the number of simple walks.} \label{table5}
\end{table}

%
\end{document}